\documentclass{amsart}    
\usepackage{amssymb,amstext,amsmath,amscd,amsthm,amsfonts,enumerate,graphicx,latexsym,stmaryrd,multicol}
\usepackage{color}
\usepackage{hyperref} 
\usepackage{tikz-cd}  

\usepackage{amsmath, mathtools}
\hypersetup{colorlinks=true} 

\usepackage[all]{xy}

\newcommand{\old}[1]{{\color{red} #1}}
\newcommand{\new}[1]{{\color{blue} #1}}
\newtheorem{thm}{Theorem}[section]  
\newtheorem{lem}[thm]{Lemma}
\newtheorem{prop}[thm]{Proposition}
\newtheorem{cor}[thm]{Corollary}
\theoremstyle{definition}
\newtheorem{dfn}[thm]{Definition}

\newtheorem{setup}[thm]{Setup}
\newtheorem{rem}[thm]{Remark}

\newtheorem{conj}[thm]{Conjecture}
\newtheorem{ex}[thm]{Example}
\newtheorem{chunk}[thm]{}
\theoremstyle{remark}

\newtheorem*{ac}{Acknowlegments}

\numberwithin{equation}{thm}

\def\depth{\operatorname{depth}}

\def\Tr{\operatorname{Tr}}
\def\Ext{\operatorname{Ext}} 

\def\m{\mathfrak{m}}

\def\ann{\operatorname{ann}} 
\def\V{\operatorname{V}}

\def\supp{\operatorname{Supp}}
\def\p{\mathfrak{p}} 
\def\Hom{\operatorname{Hom}} 
\def \Ht {\operatorname{ht}} 
\def \spec {\operatorname{Spec}}
\def\syz{\Omega}
 
\def \q {\mathfrak {q}}

\def \id {\operatorname{id}} 
\def \grade   {\operatorname{grade}}
\def \Min {\operatorname{Min}} 
\def \Ass {\operatorname{Ass}}

\def \Tor {\operatorname{Tor}}

\def \pd {\operatorname{pd}}
\def \soc {\operatorname{Soc}}
\def \im {\operatorname{Im}}
\def \Ker {\operatorname{Ker}}
\def \Cok {\operatorname{Coker}}
\def \End {\operatorname{End}}  
\def \n {\mathfrak {n}} 
\def \NF {\operatorname{NF}}  
\def \t {\operatorname{t}}
\def \rank {\operatorname{rank}}   
\begin{document}
\title{Vanishing of (co)homology of Burch and related submodules}   

\author{Souvik Dey}
\address{Department of Mathematics, University of Kansas, 405 Snow Hall, 1460 Jayhawk Blvd.,
Lawrence, KS 66045, U.S.A.}
\email{souvik@ku.edu}

\author{Toshinori Kobayashi}
\address{School of Science and Technology, Meiji University, 1-1-1 Higashi-Mita, Tama-ku, Kawasaki-shi, Kanagawa 214-8571, Japan}
\email{toshinorikobayashi@icloud.com}    

\thanks{Kobayashi was partly supported by JSPS Grant-in-Aid for JSPS Fellows 21J00567.}

\date{}    

\begin{abstract}    
We introduce the notion of Burch submodules and weakly $\m$-full submodules of modules over a local ring $(R,\m)$ and study their properties. One of our main results shows that Burch submodules satisfy $2$-Tor rigid and test property. We also show that over a local ring $(R,\m)$ a submodule $M$ of a finitely generated $R$-module $X$, such that either $M=\m X$ or $M (\subseteq \m X$) is weakly $\m$-full in $X$, is $1$-Tor rigid and a test module provided that $X$ is faithful (and $X/M$ has finite length when $M$ is weakly $\m$-full). As an application, we give some new class of modules and a new class of rings such that a conjecture of Huneke and Wiegand is affirmative for them. 
\end{abstract}     


\maketitle 

\section{Introduction}  

Throughout this paper, rings are assumed to be commutative and Noetherian, and modules are assumed to be finitely generated unless otherwise specified. In this paper, we introduce the notion of Burch submodules and study their homological properties with considerations of other related modules.
Our investigations show that Burch submodules have nice properties about the vanishing of certain (co)homologies, while they have a simple definition and many important examples.
Our motivation comes from the paper of Dao, Takahashi and second author \cite{burch} wherein the Burch ideals are introduced.
In that paper, they reveal that Burch ideals have many interesting properties.
Among them, we concentrate on a result on homological properties of Burch ideals which originate from Burch's result \cite[Theorem 5 (ii)]{b}.

\begin{thm}[Burch] \label{bideal}
Let $(R,\m)$ be a local ring.
Let $I$ be a Burch ideal of $R$, i.e., an ideal of $R$ with an inequality $\m I:_R\m\ne I:_R\m$. Let $M$ be an R-module. If $\Tor^R_t(R/I,M)=\Tor^R_{t+1}(R/I,M)=0$ for some positive integer $t$, then $M$ has projective dimension at most $t$.  
\end{thm}   

This theorem says an $R$-module of the form $R/I$ where $I$ is a Burch ideal satisfy 2-Tor rigid test property (see Definition \ref{df22} (10) for precise definition of $n$-Tor rigid test property for an integer $n\ge 1$).
Note that $n$-Tor rigid test properties of modules have been a frequent interest of study in commutative algebra, for instance see \cite{rigid}, \cite{cw}, and the numerous references therein.  
We define Burch submodules as a generalization of the class of Burch ideals (Definition \ref{burchdef}).
As one of our main results, we establish a generalization of Theorem \ref{bideal} concerning  Burch submodules: 

\begin{thm}[Corollary \ref{c2} and Proposition \ref{finpd3}]
Let $(R,\m)$ be a local ring, $M,X$ be $R$-modules, and $N$ be a Burch submodule of $X$. Assume either 
\begin{enumerate}[\rm(1)]
\item $\Tor^R_{t}(X/N,M)=\Tor^R_{t+1}(X/N,M)=0$ or
\item $\Tor^R_{t-1}(N,M)=\Tor^R_{t}(N,M)=0$
\end{enumerate}
for some positive integer $t$.
Then $M$ has projective dimension at most $t-1$.    
\end{thm}

The result of Levin and Vasconcelos \cite[page 316, Lemma]{lv} says that nonzero modules of the form $\m M$ for some $R$-modules $M$ satisfy the 2-Tor rigid test property.
We see that $\m M$ is a Burch submodule of $M$ (Example \ref{mN}), and hence we recover the cited result. Moreover such an observation allows us to recover and generalize the result of Iyenger and Puthenpurakal \cite[Theorem 3.2]{iy}; see Corollary \ref{9}. On the other hand, we develop the argument of Levin and Vasconcelos, and obtain 1-Tor rigid test property of $\m M$ in the case where $M$ is a faithful module. Our arguments also allow us to improve and recover a result of Iyenger and Puthenpurakal \cite[Theorem 3.1]{iy}:

\begin{thm}[Proposition \ref{t} and \ref{faithquot}]
Let $(R,\m)$ be a local ring, $M,X$ be $R$-modules, and $N$ be a submodule of $X$. Then we have the following:

\begin{enumerate}[\rm(1)]
\item Assume $N$ is faithful and either $\Tor_t^R(M,\m N)=0$ or $\Tor_{t+1}^R(M,X/\m N)=0$ for some integer $t\ge 0$.  
Then $\pd M\le t$.  

\item Assume $\Tor_t^R(M,X/\m N)=0=\Tor^R_t(M,X)$ for some integer $t\ge 1$. If $X\ne 0$ and $X/N$ has finite length, then either $\depth X=0$, or $\pd M<t$.   
\end{enumerate}  
\end{thm}


To explain another important source of Burch submodules, let us recall weakly $\m$-full ideals introduced in \cite{two}.
An ideal $I$ is called \textit{weakly $\m$-full} if the equality $\m I:_R\m=I$ holds.
It is shown in \cite[3.11]{two} that weakly $\m$-full ideal $I$ with $\depth R/I=0$ is Burch, and in \cite[Theorem 2.10]{hw} that $\m$-primary weakly $\m$-full ideal over a local ring $R$ of positive depth is 1-Tor rigid test.
To extend these results, we focus on weakly $\m$-full submodules $N$ of an $R$-module $X$, i.e., submodules satisfying the equality $\m N:_X \m=N$.
We show that a weakly $\m$-full submodule $N$ with $\depth X/N=0$ is a Burch submodule of $X$ (Lemma \ref{wb}).
On the other hand, we prove the following result on Tor rigid test property of weakly $\m$-full submodules:
\begin{thm}[Corollary \ref{wfinpd} and \ref{th411}]
Let $(R,\m)$ be a local ring, $M,X$ be $R$-modules, and $N$ be a weakly $\m$-full submodule of $X$.
Assume $X$ is faithful, $N\subseteq \m X$, $X/N$ has finite length, and either $\Tor_t^R(M,N)=0$ or $\Tor_{t+1}^R(M,X/N)=0$ for some integer $t\ge 0$. Then $\pd M\le t$.   
\end{thm}     

As applications, we explore local rings over which all modules satisfy $n$-Tor rigid property for some $n$.    
We first consider Tor rigid test property of Ulrich modules. We show that faithful Ulrich modules over a Cohen--Macaulay local ring of dimension $d\ge 1$ are  $d$-Tor rigid test (Theorem \ref{faithulhigh}). Celikbas told us that this result has some similarity to some result of \cite{clty}, however, our proofs are entirely different. A variation of Theorem \ref{faithulhigh} for Ulrich modules of full support, while testing against modules locally free on $\Ass(R)$, is also proved in Theorem \ref{ulloc}. Using our results, we first show that over Cohen--Macaulay local rings of minimal multiplicity, every module of constant rank is $n$-Tor rigid test for some $n$ (Theorem \ref{newmin}). After that, we show that over deformations (by a non-empty regular sequence) of them, every module is $n$-Tor rigid for some $n$ (Corollary \ref{defmin}). 
These results improve the results of \cite{gp} by lowering the number of consecutive vanishing.
Also we show that for Cohen--Macaulay local rings  with a surjection from a direct sum $\m^{\oplus n}$ of $\m$ for some $n\ge 1$ to the canonical dual of $\m$, syzygy of every non-free maximal Cohen--Macaulay $R$-module  is always $2$-Tor rigid test, and moreover it is $1$-Tor rigid test if it has constant rank or $R$ is Artinian, see Theorem \ref{trig} and Theorem \ref{t55} respectively and the remarks following them.  
At a glance, the existence of such a surjection is strange and seems to be difficult to approach, but we give a sufficient condition (Remark \ref{r56} and Proposition \ref{64}), and so it is easy to construct a non-trivial example. 
Using the notion of such rings, we also highly generalize \cite[Theorem 4.3]{hw} in Proposition \ref{genhw4.3}. 

The long standing conjecture of Huneke and Wiegand \cite{hw1} states that $M\otimes_R \Hom_R(M,R)$ has nonzero torsion if $M$ is a non-free module having constant rank over a one-dimensional local ring $R$.
We refer to \cite{hiw} for details.
Using our theorems on Tor rigid and test modules, we give some partial answers to the conjecture as follows.

\begin{thm}[Theorem \ref{74}]
Let $(R,\m)$ be a Cohen--Macaulay local ring of dimension one.
Let $M$ be an $R$-module having positive constant rank.
Assume one of the following conditions hold:
\begin{enumerate} [\rm (1)]
\item there exists an $R$-module $N$ such that $M\cong \m N$.
\item $M$ is a weakly $\m$-full submodule of some $R$-module $X$ such that $X/M$ has finite length.
\item $R$ is a Cohen--Macaulay local ring with a canonical module $\omega$ and there exists a surjection $\m^{\oplus n} \to \Hom_R(\m,\omega)$ for some $n$.
\end{enumerate}  
Then $M\otimes_R \Hom_R(M,R)$ has nonzero torsion if $M$ is non-free. 
\end{thm}  

As corollaries of this theorem, we recover results of Celikbas, Goto, Takahashi and Taniguchi \cite[Proposition 1.3]{hw} and Huneke, Iyengar and Wiegand \cite[Corollary 3.5]{hiw}.  

Now we explain the organization of the paper. In section 2, we collect our fundamental notations and definitions, and discuss about a generalization of the result of Levin and Vasconcelos, and also, that of Iyenger and Puthenpurakal \cite[Theorem 3.1]{iy}. In section 3, we give the definition of Burch submodules and their basic properties. We explain our results on weakly $\m$-full submodules in section 4.
Section 5 contains investigations on Tor rigid test property over some Cohen--Macaulay local rings. We also add further observations on the existence of a surjective $R$-homomorphism from a direct sum $\m^{\oplus n}$ of $\m$ for some $n\ge 1$ to a canonical dual of $\m$ in Section 6. We give a characterization of such a condition by using canonical ideals, and an application to numerical semigroups (Proposition \ref{pfn}). Section 7 is devoted to our approach to the conjecture of Huneke and Wiegand.
In Section 8, we collect some examples which complement our results.   

\section{preliminaries}   
Throughout this paper, $R$ will denote a local ring with unique maximal ideal $\m$ and residue field $k$.   

In this section, we give some preliminary propositions and some observations on Tor rigid test property of $R$-modules of the form $\m M$ for some $R$-module $M$. 
We begin with collecting some basic notations and definitions.  
\begin{dfn} \label{df22}   
   
\begin{enumerate}[ \rm(1)]
\item We denote by $\spec(R)$ the set of prime ideals of $R$.
\item For an $R$-module $X$, we denote by $\mu(X)$ the cardinality of a system of minimal generators of $X$. Also we denote by $\ell(X)$ the length of $X$ and let $\pd_R(X), \id_R(X)$ stand for projective and injective dimension of the $R$-module $X$ respectively. Let $X^*$ stand for $\Hom_R(X,R)$.   
\item Let $X$ be an $R$-module. We say that $X$ has constant rank if $X_\p$ is $R_\p$-free of a same rank for all $\p\in\Ass(R)$.
\item Let $X$ be an $R$-module.
We denote by $\supp_R(X)$ the set of prime ideals $\p$ with $X_\p\not=0$ and by $\Ass_R(X)$ the set of associated prime ideals of $X$.
Also, we denote by $\NF_R(X)$ the non-free locus of $X$, i.e., $\NF_R(X)=\{\p\in\spec(R)\mid X_\p\text{ is not $R_\p$-free}\}$. When the ring in question is clear, we drop the subscript $R$. It is well-known that $\NF_R(X)$ is a closed subset of $\spec(R)$. 
\item\label{sumin} If $\Min(R)\subseteq \supp(X)$, then $\supp(X)=\spec(R)$. Indeed, let $\p \in \spec(R)$, then there exists $\q\in \Min(R)$ with $\q\subseteq \p$. Since $\q\in \supp(X)$ and $\supp(X)$ is specialization closed, hence $\p \in \supp(X)$. 
\item Let $X$ be an $R$-module.  
A minimal free resolution $(F_n,\partial_n)$ of $X$ is a free resolution
$\cdots \to F_{n+1}\xrightarrow[]{\partial_{n+1}} F_n \xrightarrow[]{\partial_n} F_{n-1} \to \cdots \to F_0 \to 0$ of $X$ such that $\im(\partial_n)\subseteq \m F_{n-1}$ for each $n\ge 1$.
For each $n\ge 1$, we denote by $\syz^n M$ the image of $\partial_n$ and call it the $n$-th syzygy module of $M$.
For a convention, we put $\syz^0 M:=M$ and $\partial_0=0$. Note that the isomorphism class of $\syz^n M$ is independent of the choice of $(F_n,\partial_n)$. 
\item For an ideal $I$ of $R$, an $R$-module $X$ (not necessarily finitely generated) and a submodule $N$ of $X$, we always put $(N:_X I):=\{x\in X\mid Ix \subseteq N\}$.
Note that $(N:_XI)$ is always a submodule of $X$ containing $N$.
On the other hand, for $R$-modules $N\subseteq X$, we put $(N:_RX):=\{a\in R\mid aX\subseteq N\}$.
\item Let $X$ be an $R$-module.
The socle $\soc(X)$ of $X$ is defined to be the sum of simple submodules of $X$.
Then it is clear that $\soc(X)$ is equal to the set of elements of $X$ which is annihilated by $\m$.
Thus for a submodule $N$ of an $R$-module $X$, $\soc(X/N)$ is identified with $(N:_X \m)/N$.
\item Let $\partial\colon F\to G$ be a homomorphism between $R$-free modules $F$ and $G$.
If we fix $R$-basis of $F$ and $G$, $\partial$ can be viewed as a matrix $A$ with entries in $R$.
Then $I_1(\partial)$ denotes the ideal of $R$ generated by the entries of $A$.
It is well-known that $I_1(\partial)$ is independent of the choice of $A$.
Also if the cokernel of $\partial$ is $R$-free, then $I_1(\partial)$ is equal to either $R$ or $0$.
Indeed, we have an exact sequence $0\to \im (\partial) \to G\to \Cok(\partial)\to 0$. Suppose that $I_1(\partial)\not=R$.
Then $\im(\partial)\subseteq \m G$.
If $\Cok(\partial)$ is free, then we have $G\cong \im (\partial) \oplus \Cok(\partial)$, and so $\mu(G)=\mu(\im (\partial))+\mu(\Cok(\partial))$.
But, $\mu(\Cok(\partial))=\mu(G/\im (\partial))=\ell\left(\dfrac{G/\im (\partial)}{(\m G+\im (\partial))/\im (\partial)}\right)=\ell\left(\dfrac{G/\im (\partial)}{\m G/\im (\partial)}\right)=\ell(G/\m G)=\mu(G)$.
Hence we get $\mu(\im (\partial))=0$, i.e., $\im (\partial)=0$. 

\item  Given a positive integer $e\ge 1$, we say that a pair $(M, N)$ of  modules over a
local ring $R$ is $e$-rigid, provided the vanishing $\Tor_i^R(M,N)=0$ for $i=j,...,j+e-1$ (for $j\ge 1$) forces $\Tor_i^R(M,N)=0, \forall i\ge j$ (\cite{cw}). A module $M$ is called $e$-Tor rigid if the pair $(M,N)$ is $e$-Tor rigid for all $N$. We say that a module $M$ is a test module for projectivity (\cite[2.2]{rigid}) if for all $R$-module $N$, $\Tor_i^R(M,N)=0$ for all sufficiently large $i>0$ implies $\pd_R N<\infty$.  We say a module is $e$-Tor rigid test, if it is both $e$-Tor rigid, and test module for projectivity. We also abbreviate $1$-Tor rigid test modules simply as rigid test modules.  Notice that if $M$ is a module such that for any module $N$, $\Tor_i^R(M,N)=0$ implies $\pd N\le i$, then $M$ is rigid test module.

\end{enumerate}
\end{dfn}

\begin{chunk} \label{ch1}
We frequently use the fact that for a homomorphism $\partial:F\to G$ between $R$-free modules and for arbitrary modules $N\subseteq X$, we always have a commutative diagram 

\begin{center}
\begin{tikzcd}
0 \arrow[r] & F\otimes_R N \arrow[r] \arrow[d, "\partial\otimes N"] & F\otimes_R X \arrow[r] \arrow[d, "\partial\otimes X"] & F \otimes_R (X/N) \arrow[r] \arrow[d, "\partial\otimes (X/N)"] & 0\\
0 \arrow[r] & G\otimes_R N \arrow[r] & G\otimes_R X \arrow[r] & G \otimes_R (X/N) \arrow[r] & 0
\end{tikzcd}
\end{center}
with exact rows.
So in particular, $F\otimes_R N$ (resp. $G\otimes_R N$) can be regarded as a submodule of $F\otimes_R X$ (resp. $G\otimes_R X$), and we have $(\partial\otimes X)|_{F\otimes_R N}=\partial\otimes N:F\otimes_R N\to G\otimes_R N$.
A similar fact on $\Hom_R(-,-)$ also holds.
\end{chunk}  

We first record some useful lemmas concerning on vanishing of some homomorphisms.

\begin{lem} \label{l26}
Let $X$ be an $R$-module, and $\partial\colon F\to G$ be a homomorphism between $R$-free modules such that $\im(\partial)\subseteq \m G$.
Assume $\depth X=0$ (and hence $X\not=0$).
If $\partial\otimes X$ is injective, then $F=0$.
Similarly, if $\Hom_R(\partial,X)$ is injective, then $G=0$.
\end{lem}

\begin{proof}
Suppose that $\partial\otimes X$ is injective.
Then by the observation in \ref{ch1}, $\partial\otimes \soc(X)$ is also injective.
Since $\im(\partial)\subseteq \m G$, it yields that $(\partial\otimes \soc(X))(F\otimes_R \soc(X))\subseteq \m(G\otimes_R \soc(X))=0$, hence $F\otimes_R \soc(X)=0$ since $\partial \otimes \soc(X)$ is injective.
And since $\depth X=0$, that is, $\soc(X)\not=0$, it implies that $F=0$. Similarly, if $\Hom(\partial,X):\Hom_R(G,X)\to \Hom_R(F,X)$ is injective, then $\Hom(G,\soc(X))=0$, so $G=0$. 
\end{proof}    

\begin{lem} \label{ll27}
Let $M$ and $ N$ be $R$-modules, and $t\ge 1$ be an integer.
Take a minimal free resolution $(F_n,\partial_n)$ of $M$.
Assume $I_1(\partial_{t})\subseteq \ann(N)$.
Then
\begin{enumerate}[ \rm(1)]
\item If $N$ is faithful, then $\partial_{t}=0$, i.e., $\pd M< t$.
\item $\supp(N)\cap \supp(\syz^t M)\subseteq \NF(\syz^{t-1} M)$.
\item If $\supp(N)\supseteq \Min(R)$ and $M$ is locally free on $\Ass(R)$, then $\pd M<t$. 
\item If  $\grade N=0$ and $M$ has constant rank, then $\pd M<t$.  
\item If $M$ has locally projective dimension less than $t$ on $\spec(R)\smallsetminus \{\m\}$ and constant rank, then either $N$ has finite length or $\pd M<t$.
\end{enumerate}  
\end{lem}

\begin{proof}
(1) is clear.

(2): We set $\partial:=\partial_t$.
Take a prime ideal $\p\in\supp(N) \cap \supp(\syz^t M)$.
Since $\p\in \V(\ann(N))$, so $I_1(\partial_\p)\subseteq \ann(N)_\p\subseteq\p R_\p$.
Since $\im \partial_\p=(\syz^t M)_\p \ne 0$, so $\Cok \partial_\p=(\syz^{t-1}M)_\p$ cannot be $R_\p$-free by the remark in \ref{df22} (9).
It shows that $\p\in \NF(\syz^{t-1}M)$.

(3): Firstly, $\supp(N)=\spec(R)$ by \ref{df22} (5). Suppose that $\pd M\ge t$, that is, $\syz^t M\not=0$. Since $\syz^t M$ is a submodule of $F_{t-1}$, so $\Ass(\syz^t M)\subseteq \Ass(R)$.
Therefore there exists $\p\in \supp(\syz^t M)\cap \Ass(R)$.  Since $M$ is locally free on $\Ass(R)$, it follows that $(\syz^n M)_{\p}$ is $R_{\p}$-free for each $n$. In particular, $(\syz^t M)_{\p}$ is $R_{\p}$-free of positive rank and  $\p \notin \NF (\syz^{t-1} M)$. This now contradicts  (2) since $\p \in \supp(\syz^t M)=\supp(N) \cap \supp(\syz^t M)$ but $\p\notin \NF (\syz^{t-1} M)$.    
So we conclude that $\pd M<t$.

To prove (4): $\grade N=0$ implies that $\ann(N)$ consists of zero-divisors of $R$, in other words, $\ann(N)$ is contained in $\bigcup_{\p\in\Ass(R)}\p$.
By prime avoidance lemma, a prime ideal $\p\in\Ass(R)$ contains $\ann(N)$.
So there exists $\q\in \Min(N)$ such that $\q\subseteq \p$. 
The assumption that $M$ has constant rank yields that both $\syz^{t-1}M$ and $\syz^tM$ have constant rank. Now as $t\ge 1$, so $\Ass(\syz^t M)\subseteq \Ass(R)$. If $\supp(\syz^t M)\ne \emptyset$, then there exists $\p\in \Ass(\syz^t M)$, hence $\p \in \Ass(R)$, and $(\syz^t M)_{\p}$ is non-zero. As $\syz^t M$ has constant rank, hence $(\syz^t M)_{\q}$ is non-zero for every $\q \in \Ass(R)$. Consequently, $\supp(\syz^tM)\in\{\spec R,\emptyset\}$.      
Suppose that $\supp(\syz^tM)=\spec R$.
Then by (2) we have $\q\in \Ass(N)=\Ass(N)\cap \supp(\syz^t M)\subseteq \NF(\syz^{t-1}M)$.
On the other hand, since $\p\in \Ass(R)$ and $\syz^{t-1}M$ has constant rank
, it follows that $\syz^{t-1}_{R_{\p}}M_\p$  is $R_\p$-free.
Then by the containment $\q\subseteq \p$, $\syz^{t-1}_{R_{\q}}M_\q$ is $R_\q$-free, i.e., $\q\not\in  \NF(\syz^{t-1}M)$. 
Thus it shows a contradiction, and so $\supp(\syz^t M)$ must be empty.
It means the desired inequality $\pd M<t$.

(5): The condition that $M$ has locally projective dimension less than $t$ on $\spec(R)\smallsetminus \{\m\}$ says that $\NF(\syz^{t-1} M)\subseteq \{\m\}$. The condition that $M$ has constant rank implies that $\supp(\syz^t M)\in\{\spec(R),\emptyset\}$. 
If $\syz^t M\ne 0$, then $\supp(\syz^t M)=\spec(R)$, then by (2) we have $\Ass(N)=\Ass(N)\cap \supp(\syz^t M)\subseteq \{\m\}$.
In other words, $N$ has finite length.
\end{proof}

The proposition below gives a generalization of the observation given in \cite[Page 316]{lv}.

\begin{prop} \label{t}
Let $M$ and $N$ be $R$-modules, and $t\ge 0$ be an integer. 
Take a minimal free resolution $(F_n,\partial_n)$ of $M$.
Assume $\Tor_t^R(M,\m N)=0$.
Then the followings hold true:
\begin{enumerate}[ \rm(1)]
\item $\partial_{t+1}\otimes N=0$, i.e.,   $I_1(\partial_{t+1})\subseteq \ann(N)$.
\item $\partial_{t+1}\otimes (\m N)=0$ and $\partial_t\otimes (\m N)$ is injective.  
\item If $\depth (\m N)=0$, then $\pd M< t$.  
\item If $N$ is faithful, then $\pd M\le t$.
\item If $\supp(N)\supseteq \Min(R)$ and $M$ is locally free on $\Ass(R)$, then $\pd M\le t$.
\item If  $\grade N=0$ and $M$ has constant rank, then $\pd M\le t$.
\item If $M$ has locally projective dimension less than $t+1$ on $\spec(R)\smallsetminus \{\m\}$ and constant rank, then either $\m N=0$ or $\pd M\le t$.  
\item $\Tor^R_{t+1}(M,\m N)\cong \syz^{t+1}M\otimes_R (\m N)$. So if moreover $\Tor^R_{t+1}(M,\m N)=0$, then either $\pd M\le t$, or $\m N=0$.  
\end{enumerate}
\end{prop}   

\begin{proof}
(1) Assume $\Tor_t^R(M,\m N)=0$.
There is nothing to prove if $t=0$, so assume $t\ge 1$.
Let $C$ be the complex $(F_n\otimes N,\partial_n\otimes N)$.
Then $H_t(\m C)=\Tor_t^R(M,\m N)=0$.
By the proof of \cite[Lemma, p.316]{lv} (see also \cite[2.2]{power}), we get $\partial_{t+1}^C=0$, i.e., $\partial_{t+1}\otimes N=0$. 
Then it is easy to verify that the ideal $I_1(\partial_{t+1})$ is contained in $\ann(N)$.

(2): Due to \ref{ch1}, we see $\partial_{t+1}\otimes_R (\m N)=0$.
Then the vanishing of $\Tor_t^R(M,\m N)=\Ker(\partial_t\otimes \m N)$ implies that $\partial_t\otimes_R (\m N)$ is injective.   

(3) follows by part (2) and Lemma \ref{l26}.

(4), (5) and (6) follow by Lemma \ref{ll27} (1), (3) and (4) respectively.

For (7), if $\pd M\ge t+1$, then by  Lemma \ref{ll27} (5), we have $N$ has finite length. Hence $\m N$ also have finite length. If $\m N\ne 0$, then $\depth(\m N)=0$, but then (3) would imply $\pd M\le t$, contradicting what we assume. Thus, $\m N=0$.

For (8): Denote $N':=\m N$. If $t=0$, then $M\otimes_R N'=0$, so either $M=0$ or $N'=0$, and there is nothing to prove. So assume $t\ge 1$. Since $\partial_{t}\otimes N'$ is injective by part (2), so decomposing the map $\partial_t \otimes N': F_t \otimes_R N' \to F_{t-1} \otimes_R N'$ into a composition of $f: F_t \otimes_R N' \to \syz^t M \otimes_R N'$ and $g: \syz^t M \otimes_R N' \to F_{t-1} \otimes_R N'$, we see that $f: F_t \otimes_R N' \to \syz^t M \otimes_R N'$ is injective. Now the exact sequence $0\to \syz^{t+1} M \to F_t \to \syz^t M \to 0$ gives the exact sequence $0 \to \Tor_1^R(\syz^t M, N')\to \syz^{t+1} M \otimes_R N' \to F_t \otimes N' \xrightarrow{f} \syz^t M \otimes N'\to 0$. Since  $\im (\syz^{t+1} M \otimes_R N' \to F_t \otimes_R N')=\Ker (f)=0$, so we get $\syz^{t+1} M \otimes_R N' \cong \Tor_1^R(\syz^t M, N') \cong \Tor_{t+1}^R(M,N').$ So in particular, if $\Tor_{t+1}^R(M,N')=0$, then $\syz^{t+1}M\otimes_R N'=0$, i.e., either $N'=\m N=0$, or else $\syz^{t+1} M=0$, i.e., $\pd M\le t$.    
\end{proof}  

We also note here an Ext version of Proposition \ref{t}.

\begin{prop} \label{tt}
Let $M$ and $N$ be $R$-modules, and $t\ge 1$ be an integer.
Take a minimal free resolution $(F_n,\partial_n)$ of $M$.
Assume $\Ext^t_R(M,\m N)=0$.
Then the followings hold true:
\begin{enumerate}[ \rm(1)]
\item $\Hom(\partial_t,N)=0$, .i.e.,   $I_1(\partial_t)\subseteq \ann(N)$.   
\item $\Hom(\partial_t,\m N)=0$ and $\Hom(\partial_{t+1},\m N)$ is injective.
\item If $\depth (\m N)=0$, then $\pd M< t$.
\item If $N$ is faithful, then $\pd M< t$.
\item If $\supp(N)\supseteq \Min(R)$ and $M$ is locally free on $\Ass(R)$, then $\pd M< t$.
\item If  $\grade N=0$ and $M$ has constant rank, then $\pd M< t$.
\item If $M$ has locally projective dimension less than $t$ on $\spec(R)\smallsetminus \{\m\}$ and constant rank, then either $\m N=0$ or $\pd M< t$. 
\end{enumerate}
\end{prop}

\begin{proof}
(1) Assume $\Ext^t_R(M,\m N)=0$.
Denote by $C'$ the complex $(\Hom_R(F_{-n},N),\Hom_R(\partial_{-n+1},N)$.
Then one has $H_{-t}(\m C')=\Ext^{t}_R(M,\m N)=0$.
Thus by the proof of \cite[Lemma, p.316]{lv}, one gets $\partial^{C'}_{-t+1}=0$, that is, $\Hom_R(\partial_t,N)=0$. Since $\Hom_R(\partial_{t},N)$ is represented by the transpose matrix of a matrix representation of $\partial_{t}$, we have the inclusion $I_1(\partial_t)\subseteq \ann(N)$.
Proofs of (2)-(7) are obtained by parallel arguments of the corresponding part of Proposition \ref{t}.  
\end{proof}

Taking $N:=\m^n X$, for some $n\ge 1$, in Proposition \ref{faithquot} (6) below, one recovers \cite[Theorem 3.1.]{iy}. 

\begin{prop}\label{faithquot}
Let $M$ and $X$ are $R$-modules, and $N$ is a submodule of $X$.
Assume there exists an integer $t> 0$ such that $\Tor_t^R(M,X/\m N)=0$.
\begin{enumerate}[ \rm(1)]
\item $I_1(\partial_t)\subseteq \ann(N)$, where $(F_n,\partial_n)$ is a minimal free resolution of $M$.
\item If $N$ is faithful, then $\pd M< t$.
\item If $\supp(N)\supseteq \Min(R)$ and $M$ is locally free on $\Ass(R)$, then $\pd M<t$.
\item If  $\grade N=0$ and $M$ has constant rank, then $\pd M<t$.
\item If $M$ has locally projective dimension less than $t$ on $\spec(R)\smallsetminus \{\m\}$ and constant rank, then either $N$ has finite length or $\pd M<t$.
\item If $\Tor_t^R(M,X)=0$, then $(N:_RX)((\syz^t M)\otimes_R X)=0$.
If moreover $X\not=0$ and $X/N$ has finite length, then $\m^m((\syz^t M)\otimes_R X)=0$ for some $m\ge 1$, so either $\depth X=0$ or $\pd M<t$.  
\end{enumerate} 
\end{prop}

\begin{proof}
(1) Let $(F_n,\partial_n)$ be a minimal free resolution of $M$.
Put $M'=\syz^{t-1}M$.
Then there exists an exact sequence $F_t \xrightarrow[]{\partial_t} F_{t-1} \xrightarrow[]{p} M' \to 0$.
Also, we remark that $\Tor_1^R(M',X/\m N)=\Tor_t^R(M,X/\m N)=0$.
We have a commutative diagram

\begin{center}
\begin{tikzcd}
0 \ar[r] & F_t\otimes_R \m N \ar[r] \ar[d, "\partial_t\otimes \m N"] & F_t\otimes_R X \ar[r] \ar[d, "\partial_t\otimes X"] & F_t\otimes_R (X/\m N) \ar[r] \ar[d, "\partial_t\otimes (X/\m N)"] & 0\\
0 \ar[r] & F_{t-1}\otimes_R \m N \ar[r] \ar[d, "p\otimes \m N"] & F_{t-1}\otimes_R X \ar[r] \ar[d, "p\otimes X"] & F_{t-1}\otimes_R (X/\m N) \ar[r] \ar[d, "p\otimes (X/\m N)"] & 0\\
0 \ar[r] & M'\otimes_R \m N \ar[r] & M'\otimes_R X \ar[r] & M'\otimes_R (X/\m N) \ar[r] & 0
\end{tikzcd}
\end{center}

with exact rows and columns.
It yields equalities 
\begin{align*}
\im(\partial_t \otimes \m N)&=\Ker(p\otimes \m N)\\
&=(F_{t-1}\otimes_R \m N)\cap \Ker(p\otimes X)\\
&=(F_{t-1}\otimes_R \m N)\cap \im(\partial_t\otimes X)
\end{align*} of submodules of $F_{t-1}\otimes_R X$.
On the other hand, one has calculations
\begin{align*}
\im(\partial_t \otimes N) &\subseteq \m(F_{t-1}\otimes_R N)\cap \im(\partial_t\otimes X)\\
&=F_{t-1}\otimes_R \m N \cap \im(\partial_t\otimes X),
\end{align*}
and hence we obtain $\im(\partial_t \otimes N)\subseteq \im(\partial_t \otimes \m N)=\m\cdot\im(\partial_t \otimes N)$.
Therefore by Nakayama's lemma, $\im(\partial_t \otimes N)$ is a zero module, that is, $\partial_t \otimes N$ is a zero map.   

The assertions (2)-(5) follow by Lemma \ref{ll27}. (6): Now assume that $\Tor_t(M,X)=0$. Then we have a factorization $\partial_t\otimes X=\phi\psi$, where $\psi \colon F_t\otimes_R X \to \syz^t M\otimes_R X$ is surjective by right exactness of tensor products and $\phi\colon \syz^t M\otimes_R X \to F_{t-1}\otimes_R X$ is injective by $\Tor_t(M,X)=0$. It shows $\im(\partial_t\otimes X)\cong \im(\psi) \cong \syz^t M\otimes_R X$. Note that $\im(\partial_t\otimes N)$ may be considered as a submodule of $\im(\partial_t\otimes X)$.
Using (1), we have calculations
\[
(N:_R X)\im(\partial_t\otimes X)=\im(\partial_t\otimes((N :_RX)X))\subseteq \im(\partial_t\otimes N)=0.  
\]
Thus one gets $(N:_RX)(\syz^t M\otimes_R X)=0$.
If $X/N$ has finite length, then it means that $N:_RX$ contains some power of $\m$.
We already saw that $\syz^t M\otimes_R X$ is isomorphic to the submodule $\im(\partial_t\otimes X)$ of $F_{t-1}\otimes_R X$. Then it shows that either $\depth X=0$ or $\syz^t M\otimes_R X=0$. In latter case, we have $\syz^t M=0$, i.e., $\pd M<t$. 
\end{proof} 

We also give an Ext version of Proposition \ref{faithquot}.

\begin{prop} \label{ttt}
Let $M$ and $X$ be $R$-modules, and $N$ is a submodule of $X$.
Assume there exists an integer $t\ge 1$ such that $\Ext^t_R(M,X/\m N)=0$.
Then
\begin{enumerate}[\rm(1)]
\item $I_1(\partial_t)\subseteq \ann(N)$, where $(F_n,\partial_n)$ is a minimal free resolution of $M$.
\item If $N$ is faithful, then $\pd M< t$.
\item If $\supp(N)\supseteq \Min(R)$ and $M$ is locally free on $\Ass(R)$, then $\pd M<t$.
\item If  $\grade N=0$ and $M$ has constant rank, then $\pd M<t$.
\item If $M$ has locally projective dimension less than $t$ on $\spec(R)\smallsetminus \{\m\}$ and constant rank, then either $N$ has finite length or $\pd M<t$. 
\end{enumerate}  
\end{prop}

\begin{proof}
The assertion (1) follows by a similar argument with the previous proposition.  
The assertions (2)-(5) follow by Lemma \ref{ll27}.
\end{proof}

\begin{rem} As immediate consequences of Proposition \ref{faithquot}(2), \ref{tt}(4) and \ref{ttt}(2) we may recover the result of Celikbas--Kobayashi \cite[2.14]{ot}; If $I$ is an ideal of $R$ containing a non zero-divisor of $R$ (hence $I$ is a faithful $R$-module) and either $\Tor_t^R(M,R/\m I)=0$, or $\Ext^t_R(M,\m I)=0$, or $\Ext^{t}_R(M,R/\m I)=0$, then $\pd_R M\le t$.  
\end{rem} 

We add some remarks on the assumption of Proposition \ref{t}, \ref{tt}, \ref{faithquot} and \ref{ttt} below.
First we recall the notion of trace modules.

\begin{chunk} 
Given $R$-modules $X$ and $M$, the \textit{trace (module)} $\tau_M(X)$ of $M$ in $X$ is defined to be the submodule $\sum_{f\in \Hom_R(M,X)}\im(f)$ of $X$.
Then there is an \textit{evaluation map} $\mathrm{ev}\colon M\otimes_R \Hom_R(M,X)\to X$ given by a correspondence $x\otimes f\mapsto f(x)$, and it is easy to see that $\im(\mathrm{ev})=\tau_M(X)$.
In particular, $\tau_M(X)$ commutes with $\m$-adic completions and localizations since hom sets and tensor products commute with them.
Also note that $\tau_M(X)=X$ if and only if there exists $n\ge 0$ and a surjective homomorphism $M^{\oplus n} \to X$ (\cite[Lemma 5.1]{kt}).
\end{chunk}   

\begin{rem} \label{r25}
An $R$-module $M$ is isomorphic to $\m N$ for some $R$-module $N$ if and only if there exists an integer $n\ge 0$ and a surjective homomorphism $\m^{\oplus n}\to M$, in other words, $\tau_\m(M)=M$.
Indeed, the ``only if" part is clear.
Suppose that we have a surjection $f\colon\m^{\oplus n}\to M$.
Then we get a submodule $C\subseteq \m^{\oplus n}\subseteq R^{\oplus n}$ and isomorphisms $M\cong \dfrac{\m^{\oplus n}}{C}= \m\left(\dfrac{R^{\oplus n}}{C}\right)$.
\end{rem}

The following lemma gives some natural class of examples of faithful modules.

\begin{lem} \label{l27}
Let $M$ be an $R$-module.
If $M$ is locally faithful at associated primes of $R$ (e.g. $M$ is locally free at associated primes of $R$ and $\supp N=\spec R$), then $M$ is faithful.
\end{lem}  

\begin{proof}
To show $\ann_R(M)=0$, we only need to check that $\ann_R(M)_\p=0$ for all $\p\in \Ass(\ann_R(M))$.
Since $\Ass(\ann_R(M))\subseteq \Ass(R)$, the assertion follows by the assumption.  
\end{proof}

\begin{rem}\label{fflat} The property of being faithful remains preserved under flat extension. Indeed, an $R$-module $M$ is faithful if and only if there is an injective map $R\to \End_R(M)$ if and only if there is an injective map $S\to \End_R(M)\otimes_R S\cong \End_S(M\otimes_R S)$ for every flat extension $S$ of $R$.   
\end{rem}

If we work on a Cohen--Macaulay ring with a canonical module, then we can detect the locally faithful property by some trace module.
Before explaining the statement, we prepare some notations.

\begin{chunk} \label{214}
Let $R$ be a Cohen--Macaulay local ring with a canonical module $\omega$.
Then for an $R$-module $M$, we denote by $M^\dag$ the canonical dual $\Hom_R(M,\omega)$.
Note that if $M$ is maximal Cohen--Macaulay then so does $M^\dag$.
For an integer $i>0$ and maximal Cohen--Macaulay $R$-modules $M$ and $N$, one has an isomorphism $\Ext^i_R(M,N)\cong \Ext^i_R(M^\dag,N^\dag)$. 
\end{chunk}

The following Lemma will be used in the proof of Theorem \ref{faithulhigh}. 

\begin{lem} \label{fs}
Let $R$ be a Cohen--Macaulay local ring with a canonical module $\omega$.
Let $M$ be an $R$-module.
We define $\sigma^R(M)$ to be the isomorphism class of $\omega/\tau_M(\omega)$.
\begin{enumerate}[\rm(1)]
\item $M$ is faithful if and only if so is $M^\dag$.
\item Assume $R$ is Artinian.
Then $M$ is faithful if and only if $\sigma^R(M)=0$.
\item $M$ is faithful if and only if $\sigma^R(M)$ locally vanishes on $\Ass(R)$.

Assume $M$ is maximal Cohen--Macaulay.
Then
\item Let $x\in \m$ be a non zerodivisor.
Then $\sigma^R(M)\otimes_R R/xR=\sigma^{R/xR}(M/xM)$.
\item Assume $\dim R\ge 2$ and $M$ is faithful.
Take a non zerodivisor $x\in \m$ such that $x\not\in\p$ for any $\p \in \Min(\sigma^R(M))\smallsetminus\{\m\}$.
Then the $R/xR$-module $M/x M$ is faithful.
\end{enumerate}
\end{lem}  

\begin{proof} 
(1) The inclusion $\ann M\subseteq \ann M^\dag$ is obvious.
Thus if $M^\dag$ is faithful, then so is $M$.
Conversely, if $M$ is faithful, then $M$ is locally faithful on $\Ass(R)$.
For each $\p\in \Ass(R)$, $R_\p$ is Artinian, so we have $M_\p\cong (M_\p)^{\dag\dag}$ and hence $(M_\p)^\dag$ is faithful.
Since $(-)^\dag$ commutes with localizations, $M^\dag$ is locally faithful on $\Ass(R)$, which implies that $M^\dag$ is faithful by Lemma \ref{l27}.
(2) If $M$ is faithful, then $M^{\dagger}$ is faithful, then an injective map $R\to (M^{\dagger})^{\oplus n}$ is defined by assigning the generators of $M^{\dagger}$. Taking $\omega$-dual and remebering $M$ is maximal Cohen--Macaulay (since $R$ is Artinian) we see that $M^{\oplus n}$ surjects onto $\omega_R$.  Conversely, if there is a surjection $M^{\oplus n} \to \omega_R$, then there is an injection $R \to (M^{\dagger})^{\oplus n}$, so $M^{\dagger}$ is faithful, so $M$ is faithful.   
(3) It follows by (2) and Lemma \ref{l27} and since trace commutes with localizations.
(4) is a consequence of the fact: an isomorphism $\Hom_R(M,\omega)\otimes_R R/xR \cong \Hom_{R/xR}(M/xM,\omega_{R/xR})$ is given by a natural mapping. 
(5) By (3), it is enough to check that $\sigma^{R/xR}(M/xM)$ locally vanishes on $\Ass(R/xR)$.
Then by (4), we need to show $\Ass_R(R/xR)\cap\supp(\sigma^R(M))=\emptyset$.
If not, then there exists $\p \in \Ass_R(R/xR)\cap\supp(\sigma^R(M))$, so $x\in \p$, so $\depth (R/xR)_{\p}=0$, so $\depth R_{\p}=1=\Ht \p$  (hence $\p\ne \m$). Since $M$ is faithful, by (3) we have $\Ass(R)\cap \supp(\sigma^R(M))=\emptyset$, hence $\p \in \supp(\sigma^R(M))$ and $\Ht \p=1$ implies $\p \in \Min (\sigma^R(M))$, but this contradicts $x\not\in\p$ for any $\p \in \Min(\sigma^R(M))\smallsetminus\{\m\}$. Thus,  $\Ass_R(R/xR)\cap\supp(\sigma^R(M))=\emptyset$. 
\end{proof} 

\section{Burch submodules}  

In this section, we naturally extend the definition of Burch ideals \cite[Definition 2.1]{burch} to that of Burch submodules, and study some of their basic properties.

\begin{dfn}\label{burchdef}
Let $X$ be an $R$-module.
An $R$-submodule $N$ of $X$ is said to be \textit{Burch} if $\m(N:_X\m)\ne \m N$, i.e., $\m(N:_X\m)\not\subseteq \m N$. 
Note that since $\m(0:_X \m)=0$, so the $0$ submodule is never Burch. 
\end{dfn}

We note that Burch submodules of $R$ is exactly the Burch ideals in the sense of \cite[Definition 2.1]{burch}.


\begin{lem}\label{dep} Let $N$ be a Burch submodule of an $R$-module $X$. Then $\soc(X/N)\ne 0$, i.e., $\depth_R(X/N)=0$.  
\end{lem}  

\begin{proof}
If $\soc(X/N)=0$, then $(N:_X \m)=N$, and then $\m N=\m(N:_X \m)$, so $N$ is not Burch.   
\end{proof}

The following alternative characterization of Burch submodules will be used frequently without further reference.  

\begin{lem}\label{burchalter}
Let $X$ be an $R$-module.
A submodule $N$ of $X$ is Burch if and only if $(\m N:_X \m)\ne (N:_X \m)$.
\end{lem}

\begin{proof} If $(\m N:_X \m)= (N:_X \m)$, then $\m (\m N:_X \m)=\m(N:_X \m)$. But we always have $N\subseteq (\m N:_X \m)$, hence $\m N \subseteq \m (\m N:_X \m)\subseteq \m N$. Thus $\m (N:_X \m)=\m N$, and $N$ is not Burch.   

Conversely, if $N$ is not Burch, then $\m N=\m (N:_X \m)$.
Then, $(N:_X \m)\subseteq (\m N:_X \m)$.
Since the reverse inclusion is obvious, we get $(N:_X \m)=(\m N:_X \m)$.  
\end{proof}    

The following is a straightforward generalization of \cite[Example 2.2(2)]{burch}.

\begin{ex}\label{mN} Let $N$ be a submodule of an $R$-module $X$.
If $\m N\ne 0$, then $\m N$ is a Burch submodule of $X$.
Indeed, $\m(\m N:_X \m)=\m N$, and since $\m N\ne 0$, so $\m N\ne \m(\m N)$, thus $\m N$ is Burch.    
\end{ex}  

The following proposition shows that if $M$ is a Burch submodule of some module, then $\Ext^1_R(k,M)\ne 0$.   

\begin{prop} Let $M$ be a submodule of an $R$-module $X$ such that $\Ext^1_R(k,M)=0$. Then, $(M:_X \m)=M+\soc(X)$. In particular, $M$ is not a Burch submodule of $X$. 
\end{prop}   

\begin{proof}
We always have $M+\soc(X)\subseteq (M:_X \m)$, so just need to prove the other inclusion.
Let if possible there exists $f\in (M:_X \m)\smallsetminus (M+\soc(X))$. Since $f\notin M$, so $M+Rf\ne M$, so $(M+Rf)/M\ne 0$.
Moreover, $\m f\subseteq M$, so $(M+Rf)/M$ is annihilated by $\m$ hence, $(M+Rf)/M\cong k^{\oplus n}$ is a non-zero $k$-vector space. 
Since $\Ext^1_R(k,M)=0$, so the exact sequence $0\to M \to M+Rf \to k^{\oplus n}\to 0$ splits, hence $M+Rf=M\oplus N$, where $N$ is an $R$-submodule of $X$ and $N\cong k^{\oplus n}$, hence $\m N=0$, i.e., $N\subseteq (0:_X \m)=\text{Soc}(X)$.
So, $f\in M+Rf\subseteq M+\text{Soc}(X)$, contradicting $f\notin M+\soc(X)$.   
Thus, we must have $(M:_X \m)\subseteq M+\soc(X)$.
Moreover, since we always have $M+(0:_X \m)\subseteq (\m M:_X \m)$, so we also get $(M:_X \m)\subseteq (\m M:_X \m)$, hence $M$ is not Burch in $X$. 
\end{proof}  

The following lemma improves and generalizes one direction of \cite[Lemma 2.11]{burch}.  

\begin{lem}\label{burchsum}
Let $N$ be a Burch submodule of an 
$R$-module $X$. 
Then there exists an element $a\in \m$ such that $k$ is a direct summand of $N/aN$.
If moreover $\depth N>0$, then that $a\in \m$ can be chosen to be $N$-regular, and moreover such that $aN$ is a Burch submodule of $N$.  
\end{lem}  

\begin{proof}
Assume $N$ is Burch.
If $\m\subseteq (\m N:_R(N:_X \m))$, then $\m (N:_X\m)\subseteq \m N$, but then $N$ is not Burch.
Thus, $\m \not\subseteq (\m N:_R(N:_X \m))$.
Choose $a\in \m$ such that $a\notin (\m N:_R(N:_X \m))$ (and moreover, if $\depth N>0$, then $a\in \m\smallsetminus (\m N:_R(N:_X \m))$ can be chosen to be $N$-regular by prime avoidance).
Then $a(N:_X\m)\not\subseteq \m N$.
Hence, there exists $b\in (N:_X \m)$ such that $ab\notin \m N$.
Also, $ab\in \m (N:_X \m)\subseteq N$.
So, $ab\in N\smallsetminus \m N$.
Moreover, $\m b\subseteq \m(N:_X\m)\subseteq N$, thus $\m ab\subseteq aN$, so $\m\overline{ab}=0$ in $N/aN$.
Hence, we can define an $R$-linear map $f:R/\m \to N/aN$ by $f(\bar 1)=\overline{ab}$.
This is split injective since $\overline{ab}$ is part of a minimal system of generators of $N/aN$ (as $ab\in N\smallsetminus \m N$).
This shows $k$ is a direct summand of $N/aN$.

Now assume $\depth N>0$ and choose $a\in \m$ in the previous construction to be $N$-regular.
To show $aN$ is a Burch submodule of $N$, first note that $y :=ab\in N$ and $\m y=\m ab =a(\m b)\subseteq a \m (N:_X \m)\subseteq a N$.
Hence, $y\in (aN :_N \m)$, so $ay\in \m (aN:_N \m)$.
As $y\notin \m N$ and $a$ is $N$-regular, so $ay\notin a\m N=\m(aN)$.
Thus $\m(aN:_N \m)\neq \m(aN)$, hence $aN$ is Burch submodule of $N$.  
\end{proof}   

Note that if $\depth X$ were positive, then by prime avoidance, the $a\in \m$ in Lemma \ref{burchsum} could be chosen to be $X$-regular as well (in fact, by prime avoidance, $a\in \m\smallsetminus (\m N:_R(N:_X \m))$ can be chosen to avoid any finite set of non-maximal prime ideals). In view of this, the following can be seen as a partial converse to Lemma \ref{burchsum}.  

\begin{lem}\label{burchsumconv}
Let $R$ have positive depth.  
Let $N$ be an $R$-submodule of an 
$R$-module $X$.
Assume that there exists $X$-regular element $a\in \m$ such that $k$ is a direct summand of $N/aN$.
If $\depth X>1$, then $N$ is Burch submodule of $X$.  
\end{lem} 

\begin{proof}
Fix a split injection $f:R/\m\to N/aN$.
Let $c\in N$ be such that $f(\bar 1)=\overline c$.
Then $\m \overline c=\bar 0$ in $N/aN$, so $\m c\subseteq aN\subseteq aX$. So, $\bar c\in X/aX$ is annihilated by $\m$.
But $\depth X/aX=\depth X -1>0$ as $a$ is $X$-regular, $\m$ contains an $X/aX$-regular element.
Thus, $\bar c=\bar 0$ in $X/aX$, i.e., $c\in aX$.
Write $c=ab$ where $b\in X$.
Then, $\m ab=\m c\subseteq aN$.
Since $a$ is $X$-regular and $N$ is a submodule of $X$, so we get $\m b\subseteq N$, that is, $b\in (N:_X \m)$.
Since $f(\bar 1)=\bar c$, and $f$ is a split injection, $\bar c$ is part of a minimal system of generators of $N/aN$, thus $c\in N\smallsetminus \m N$.
Thus $c=ab\notin \m N$ but $ab\in \m (N:_X \m)$.
It means that $N$ is Burch.  
\end{proof}   

\begin{rem} Note that unlike \cite[Lemma 2.11]{burch}, Lemma \ref{burchsum} does not require $\depth X>1$.   
\end{rem}    

The following lemma can be regarded as an ``embedding-free" characterization of Burch submodules.

\begin{lem} \label{l310}
Let $N$ be an $R$-module.
Then the following conditions are equivalent to each other.
\begin{enumerate}[ \rm(1)]
\item There exists an $R$-module $X$ such that $N$ is isomorphic to a Burch submodule of $X$.
\item There exists a homomorphism $f\colon \m \to N$ such that $f\otimes_R k$ is nonzero.

\item $\tau_\m(N)\not\subseteq \m N$. 
 
\item The inclusion $\Hom_R(\m,\m N) \hookrightarrow \Hom_R(\m, N)$ induced by the natural inclusion $\m N \hookrightarrow N$
is not surjective. 

If moreover $\depth N>0$, then each of the conditions above is equivalent to any
of the following.   

\item There exists an element $a\in \m$ which is regular on $N$ such that $k$ is a direct summand of $N/aN$.
\item There exists an element $a\in \m$ which is regular on $N$ such that $aN$ is a Burch submodule of $N$.
\end{enumerate}
\end{lem}

\begin{proof}
The implications (2)$\iff$(3)$\iff$(4) are clear. 
(1)$\implies$(2): We may assume $N$ is a Burch submodule of $X$.
Therefore there is an element $a\in N:_X\m$ such that $\m a\not\subseteq \m N$.
Then an $R$-homomorphism $f\colon R \to X$ is given by a correspondence $1\mapsto a$.
Consider the restriction $f|_\m \colon \m \to X$ of $f$ on $\m$.
It is clear that $\im (f|_\m) \subseteq N$ and $\im (f|_\m) \not\subseteq \m N$.
Thus it gives an $R$-homomorphism $g\colon \m \to N$ such that $g\otimes_R k\not=0$.

(2)$\implies$(1): We consider a push-out diagram
$
\begin{tikzcd}
\m \ar[r, "i"] \ar[d, "f"] & R \ar[d, "g"] \\
N \ar[r, "j"] & X
\end{tikzcd}$    

of two maps $i$ and $f$, where $i$ is the natural inclusion.
Then, $X\cong \dfrac{N\oplus R}{\{(-f(y),i(y)):y\in \m\}}$, and $j(n):=\overline{(n,0)},\forall n\in N$. As $i$ is injective, so 
$j$ is injective, hence we may regard $N$ as a submodule of $X$.
Set an element $x=g(1)\in X$.
Then we have equalities $\m\cdot x=g(\m\cdot 1)=(jf)(\m)$.
It means that $\m\cdot x\subseteq N$.
On the other hand, the assumption $f\otimes_R k\not=0$ says that $g(a)\not\in \m N$ for some $a\in \m$.
It then implies that $a\cdot x=g(a)\not\in \m N$.
We achieve the inequality $\m N\not= \m(N:_X\m)$.

Now assume $\depth N>0$.
(1)$\implies$(5): It follows by Lemma \ref{burchsum} (and remembering that if $N\cong N'$, then $N/aN\cong N'/aN'$). (5)$\implies$(6): Take an element $a\in \m$ regular on $N$ such that $k$ is a direct summand of $N/aN$.
There is an element $x\in N$ whose image in $N/aN$ corresponds to $1\in k$ of the summand.  
Then $x\in N\smallsetminus \m N$ and $\m\cdot x\subseteq aN$, so $x\in (aN:_N \m)$.
In particular, $ax$ belongs to $\m(aN:_N\m)\smallsetminus a\m N$ ($ax\notin a\m N$ because $a$ is $N$-regular). It shows that $\m(aN:_N\m)\not\subseteq \m (aN)$, i.e., $aN$ is a Burch submodule of $N$. (6)$\implies$(1): This is trivial.    
\end{proof}   

\begin{rem} If $f:\m \to N$ is a homomorphism such that the map of $k$-vector spaces $f\otimes k : \m \otimes_R k \to N \otimes_R k$ is non-zero, then its $k$-dual is non-zero, and then by Tensor-hom adjunction, the natural map 

$\Hom(f,k): \Hom_R(N,k)\cong \Hom_k(N\otimes_R k, k) \to \Hom_R(\m ,k)\cong \Hom_k(\m \otimes_R k, k)$ is non-zero. This implies that $N$ as in Lemma \ref{l310} has extremal complexity and curvature by \cite[Theorem 8]{a}.   
\end{rem}

The following result will be used in the proof of Lemma \ref{l521}.  

\begin{cor}\label{313} Let $R$ have positive depth. Let $M,N$ be $R$-modules of positive depth. Then, the following holds:

\begin{enumerate}[\rm(1)]
\item  If $M$ is a Burch submodule of some $R$-module, $\depth_R N=1$ and $\Ext^1_R(M,N)=0$, then $\Hom_R(M,N)$ is isomorphic to a Burch submodule of some $R$-module.  
\item If $R$ is Cohen--Macaulay, admits a canonical module $\omega$ and $M$ is maximal Cohen--Macaulay and a Burch submodule of some $R$-module, then $M^\dagger$ is isomorphic to a Burch submodule of some $R$-module (and $\dim R=1$).    
\item If $N$ be a Burch submodule of some $R$-module, $M\not=0$ and $\Ext^1_R(M,N)=0$, then $\Hom_R(M,N)$ is isomorphic to a Burch submodule of some $R$-module.  
\end{enumerate}
\end{cor}  

\begin{proof} (1) By Lemma \ref{burchsum} and its proof (and the discussion following it), there exists $a\in \m$ which is regular on $R, M$ and $N$ such that $k$ is a direct summand of $M/aM$. Write $M/aM\cong k\oplus M'$ for some $R/aR$-module $M'$.
The short exact sequence $0\to N \xrightarrow{a} N\to N/aN\to 0$ gives an exact sequence $0\to \Hom_R(M,N)\xrightarrow{a} \Hom_R(M,N)\to \Hom_R(M,N/aN)\to \Ext^1_R(M,N)(=0)$. Hence, $\Hom_R(M,N)/a\Hom_R(M,N)\cong \Hom_R(M,N/aN)$. By \cite[Lemma 2(ii), page 140]{mat}, we get
\[
\Hom_R(M,N/aN)\cong \Hom_{R/aR}(M/aM,N/aN)\cong \Hom_{R/aR}(k,N/aN)\oplus \Hom_{R/aR}(M',N/aN).
\] Since $\depth_R N=1$, so $\depth_{R/aR} N/aN=0$, so $\Hom_{R/aR}(k,N/aN)$ is a non-zero $k$-vector space. It shows that $k$ is a direct summand of $\Hom_R(M,N)/a\Hom_R(M,N)$.
Since $\Hom_R(M,N)$ can be embedded into $N^{\oplus \mu(M)}$, $a$ is also regular on $\Hom_R(M,N)$, hence $\Hom_R(M,N)$ is a Burch submodule of some $R$-module by Lemma \ref{l310}. 

(2) By Lemma \ref{burchsum} and its proof (and the discussion following it), there exists $a\in \m$ which is regular on both $R$ and $M$, such that $k$ is a direct summand of $M/aM$, hence $\depth_{R/aR} M/aM=0$, so $\depth_R M=1$. Since $M$ is maximal Cohen--Macaulay, so $\dim R=\depth_R M=1$. Hence $\depth_R \omega=1$. The claim now follows by part (1) since $\Ext^1_R(M,\omega)=0$.  

(3) By Lemma \ref{burchsum} and its proof (and the discussion following it), there exists $a\in \m$ which is regular on $R, M$ and $N$ such that $k$ is a direct summand of $N/aN$. Write $N/aN\cong k\oplus N'$ for some $R$-module $N'$. The short exact sequence $0\to N \xrightarrow{a} N\to N/aN\to 0$ gives an exact sequence $0\to \Hom_R(M,N)\xrightarrow{a} \Hom_R(M,N)\to \Hom_R(M,N/aN)\to \Ext^1_R(M,N)(=0)$. Hence we get 
\[
\Hom_R(M,N)/a\Hom_R(M,N)\cong \Hom_R(M,N/aN)\cong \Hom_R(M,k)\oplus \Hom_R(M,N').
\]
Since $M\not=0$, $\Hom_R(M,k)$ is a nonzero $k$-vector space.
It shows that $k$ is a direct summand of $\Hom_R(M,N)/a\Hom_R(M,N)$.
Since $\Hom_R(M,N)$ can be embedded into $N^{\oplus \mu(M)}$, $a$ is also regular on $\Hom_R(M,N)$, hence $\Hom_R(M,N)$ is a Burch submodule of some $R$-module by Lemma \ref{l310}.    
\end{proof}

In view of Example \ref{mN}, Propositions \ref{t} and  \ref{faithquot} tell us that some special type of Burch submodules and their quotients fulfils the rigid test property.  So we next investigate test module properties of Burch submodules (and quotients by them) of a module in general.    

\begin{prop}\label{finpd2}
Let $N$ be a Burch submodule of an $R$-module $X$. Let $t\ge 1$ be an integer, and $M$ be an $R$-module such that 
\begin{enumerate}[\rm(1)]
\item $\Tor^R_t(M,X/N)=0$, and
\item the map $\Tor^R_{t+1}(M,X)\to \Tor^R_{t+1}(M,X/N)$ induced by the canonical surjection $X\to X/N$ is surjective.
\end{enumerate} Then, it holds that $\pd M<t$. 
\end{prop}   

\begin{proof}   
Let $(F_n,\partial_n)$ be a minimal free resolution of $M$. Let us put, for the rest of the proof,  $N':=X/N$. Note that for $i\ge 1$, we have $\partial_i(F_i)\subseteq \m F_{i-1}$. So in particular, $(\partial_t\otimes X)(F_t\otimes_R (N:_X \m))\subseteq F_{t-1} \otimes \m (N:_X\m)\subseteq F_{t-1}\otimes_R N$.
Hence $(\partial_t \otimes N')(F_t \otimes_R \soc_R(N'))=0$ in $F_{t-1}\otimes_R N'$.
Thus, $F_t \otimes_R \soc_R(N')\subseteq \Ker (\partial_t \otimes N')=\im(\partial_{t+1}\otimes N')$, where the last equality holds since $\Tor^R_t(M,N')=0$.
Take $a\in F_t \otimes_R (N:_X \m) \subseteq F_t \otimes_R X$. 
Then, $\bar a \in F_t \otimes_R \soc_R(N')=\im(\partial_{t+1}\otimes N')$.
So, there exists $b\in F_{t+1}\otimes_R X$ such that $\bar a=(\partial_{t+1}\otimes N')(\bar b)$ (here $\overline{(-)}$ denotes images in $(F_i \otimes_R X)/(F_i\otimes_R N)\cong F_i\otimes_R N'$).
Thus, $a-(\partial_{t+1}\otimes X)(b) \in F_t\otimes_R N$ $(*)$.
Since $\bar a \in F_t \otimes_R \soc_R(N')$, that is, $\m \bar a=0$ in $F_t\otimes N'$, it says $(\partial_{t+1}\otimes N')(\m \bar b)=0$, i.e., $\m \bar b \in \Ker (\partial_{t+1} \otimes N')$.
Now, the assumption (2) says that we have surjections $\Ker(\partial_{t+1} \otimes X)\to \dfrac{\Ker(\partial_{t+1} \otimes X)}{\im (\partial_{t+2} \otimes X)}\to \dfrac{\Ker(\partial_{t+1} \otimes N')}{\im (\partial_{t+2} \otimes N')}$.
So, $\m \bar b\subseteq \Ker(\partial_{t+1} \otimes N') \subseteq \overline{\Ker(\partial_{t+1} \otimes X)}+\im (\partial_{t+2} \otimes N')$.
Lifting the inclusions to $F_{t+1}\otimes_R X$ one obtains
$\m b \subseteq \Ker(\partial_{t+1} \otimes X)+\im(\partial_{t+2} \otimes X)+ F_{t+1}\otimes_R N\subseteq \Ker(\partial_{t+1} \otimes X)+ F_{t+1}\otimes_R N$.
Then one has inclusions
$\m(\partial_{t+1}\otimes X)(b)=(\partial_{t+1}\otimes X)(\m b)\subseteq (\partial_{t+1}\otimes X)(F_{t+1}\otimes_R N)=(\partial_{t+1} \otimes N)(F_{t+1}\otimes_R N)\subseteq \m F_t \otimes_R N=F_t\otimes_R \m N$. Hence, $(\partial_{t+1}\otimes X)(b)\in F_t \otimes (\m N:_X \m)$. This along with the equation $(*)$ yields $a\in (\partial_{t+1}\otimes X)(b)+F_t\otimes_R N\subseteq F_t \otimes (\m N:_X \m)+F_t\otimes_R N=F_t\otimes_R (\m N:_X\m)$.
Since $a\in F_t\otimes_R (N:_X \m)$ was arbitrary we get $F_t \otimes_R (N:_X \m)\subseteq F_t\otimes_R (\m N:_X\m)$.
Since $(\m N:_X \m)\subseteq (N:_X \m)$ always holds, so we also have $F_t\otimes_R (\m N:_X \m)\subseteq F_t\otimes_R (N:_X \m)$.
Thus, $F_t\otimes_R (\m N:_X \m)= F_t\otimes_R (N:_X \m).$  So, if the free module $F_t$ is non-zero, then we get $(\m N:_X \m)=(N:_X \m)$, contradicting $N$ is a Burch submodule of $X$ (see Lemma \ref {burchalter}).  Thus we must have $F_t=0$, i.e., $\pd M<t$.
\end{proof}    

We obtain the following two immediate corollaries.
The latter one generalize Burch's result \cite[Theorem 5 (ii)]{b} by letting $X=R$.

\begin{cor} \label{c1}
Let $N$ be a Burch submodule of an $R$-module $X$. Let $t\ge 1$ be an integer, and $M$ be an $R$-module such that $\Tor_t(M,N)=\Tor_t(M,X/N)=0$.
Then $\pd M<t$.
\end{cor}

\begin{proof} The exact sequence $\Tor^R_{t+1}(M,X)\to \Tor^R_{t+1}(M,X/N)\to \Tor^R_t(M,N)=0$ shows that the map $\Tor^R_{t+1}(M,X)\to \Tor^R_{t+1}(M,X/N)$ is surjective hence we can apply Proposition \ref{finpd2}.     
\end{proof}

\begin{cor}\label{c2}
Let $N$ be a Burch submodule of an $R$-module $X$. Let $t\ge 1$ be an integer, and $M$ be an $R$-module such that $\Tor_t(M,X/N)=\Tor_{t+1}(M,X/N)=0$.
Then $\pd M<t$.
\end{cor}  

\begin{proof}
Since $\Tor_{t+1}(M,X/N)=0$, the required map $\Tor_{t+1}^R(M,X)\to \Tor_{t+1}(M,X/N)$ is automatically surjective.
Hence we may apply Proposition \ref{finpd2}.
\end{proof}

The following Ext version of Proposition \ref{finpd2} follows by a parallel argument of Proposition \ref{finpd2}.  

\begin{prop}
Let $N$ be a Burch submodule of an $R$-module $X$.
Let $t\ge 1$ be an integer, and $M$ be an $R$-module such that 
\begin{enumerate}[\rm(1)]
\item $\Ext^{t+1}_R(M,X/N)=0$, and
\item the map $\Ext_R^t(M,X)\to \Ext_R^{t}(M,X/N)$ induced by the canonical surjection $X\to X/N$ is surjective.
\end{enumerate} 
Then, it holds that $\pd M\le t$.  
\end{prop}

The following result shows that Burch submodules are 2-Tor-rigid test. In view of Example \ref{mN}, Proposition \ref{finpd3}(1) generalizes part of Proposition \ref{t}(8).    

\begin{prop}\label{finpd3}
Let $N$ be a Burch submodule of an $R$-module $X$.
Let $t\ge 1$ be an integer and $M$ an $R$-module such that either of the following conditions hold:
\begin{enumerate}[\rm(1)]
\item $\Tor^R_{t}(M,N)=\Tor^R_{t-1}(M,N)=0$, or
\item $\Ext^{t+1}_R(M,N)=\Ext^t_R(M,N)=0$.
\end{enumerate}
Then $\pd M<t$.
\end{prop} 

\begin{proof}  
(1): If $t=1$, then $M\otimes_R N=0$, so $M=0$ as $N\ne 0$ since $N$ is Burch. So, we may assume $t\ge 2$. Consider a minimal free resolution $(F_n,\partial_n)$ of $M$. Set $N':=N:_X \m$. So we have $\m N'\subseteq N\subseteq N'$. Then we get
\begin{align*}
\im (\partial_t\otimes N') & \subseteq \Ker (\partial_{t-1} \otimes N') \cap \m(F_{t-1} \otimes_R N') \\
& = \Ker (\partial_{t-1} \otimes N') \cap (F_{t-1}\otimes_R \m N') \\
& \subseteq \Ker (\partial_{t-1} \otimes N') \cap (F_{t-1}\otimes_R N) \\
& \subseteq \Ker (\partial_{t-1} \otimes N)\\
& = \im (\partial_t\otimes N).
\end{align*}  
Here the equality in fifth line follows from the assumption $\Tor_{t-1}^R(M,N)=0$.
Therefore, for any element $x\in F_t\otimes_R N'$, there exists $y\in F_t\otimes_R N$ such that $(\partial_t\otimes N')(x)=(\partial_t \otimes_R N)(y)=(\partial_t \otimes N')(y)$.  
It means that $x-y$ is in $\Ker(\partial_t\otimes N')$.
Then since $\m(x-y)\subseteq \m(F_t\otimes_R N')+\m(F_t\otimes_R N)\subseteq F_t\otimes N$, we have     
\begin{align*} 
\m x =\m(x-y)+\m y  &\subseteq \left(\Ker(\partial_t\otimes N'\right)\cap (F_t\otimes_R N)) +\m (F_t\otimes_R N)\\
& \subseteq \Ker(\partial_t\otimes N)+F_t\otimes_R \m N\\
& = \im(\partial_{t+1} \otimes N)+F_t\otimes_R\m N\\
& \subseteq F_t\otimes_R \m N.
\end{align*}
Here the equality in third line follows from the assumption $\Tor_t^R(M,N)=0$.
It follows that $\m(F_t\otimes_R N')\subseteq F_{t}\otimes_R \m N$, i.e., $F_t\otimes_R \m N'\subseteq F_t \otimes_R \m N$, hence $F_t\otimes_R \m N=F_t \otimes_R \m N'$. 
Suppose that $F_t\not=0$. 
Then it implies that $\m N'=\m N$, a contradiction to $N$ being Burch.
Thus we obtain $F_t=0$, which shows that $\pd M<t$.

(2): The assertion (2) can be shown by a parallel argument of (1).  
\end{proof}   

As a consequence of our previous results, we can derive the following corollary, part (2) and (3) of which recovers \cite[Theorem 3.2]{iy} and \cite[Lemma, page 316]{lv}.  

\begin{cor}\label{9}
Let $N$ be a submodule of an $R$-module $X$. Let $t\ge 1, n\ge 1$ be integers, such that $\m^n N\ne 0$. Let $M$ be an $R$-module such that one of the following conditions hold:   
\begin{enumerate}[\rm(1)]
\item $\Tor^R_t(M,X/\m^nN)=\Tor^R_{t+1}(M,X/\m^n N)=0$,
\item  $\Tor^R_t(M,X/\m^nN)=\Tor^R_t(M,\m^n N)=0$, or
\item $\Tor^R_{t-1}(M,\m^n N)=\Tor^R_t(M,\m^n N)=0$.
\end{enumerate} 
Then, it holds that $\pd M< t$.      
\end{cor}   

\begin{proof} Since $\m^n N\ne 0$, so $\m^n N$ is a Burch (Example \ref{mN}) submodule of $X$. 
Thus, the claim follows from Corollary \ref{c1}, \ref{c2}, and Proposition \ref{finpd3} (or Proposition \ref{t}(8)).  
\end{proof}    

We add here a short proposition, which tells us that Burch submodules of positive depth can be used to test finiteness of injective dimensions of modules via some vanishings of Ext.   

\begin{prop} \label{321}
Let $N$ be a Burch submodule of an $R$-module $X$ such that $\depth N>0$.
Let $M$ be an $R$-module.
Assume there exists an integer $t\ge \depth M$ such that $\Ext^t_R(N,M)=\Ext^{t-1}_R(N,M)=0$.
Then $\id M<\infty$.
\end{prop}   

\begin{proof} We may assume $M\ne 0$. It follows from Lemma \ref{burchsum} that $N/aN$ has the residue field as a direct summand for some $N$-regular element $a\in R$. From the exact sequence $0\to N\xrightarrow{\cdot a} N \to N/aN\to 0$, we see an exact sequence $\Ext^{t-1}_R(N,M) \to \Ext^t_R(N/a N,M) \to \Ext^t_R(N,M)$. Thus by our assumption, we obtain $\Ext^t_R(N/aN,M)=0$, and so $\Ext^t_R(k,M)=0$. As $t\ge \depth M$, so $t>\depth M$. Then we may apply \cite[II. Theorem 2]{p} to obtain $\id M<\infty$.  
\end{proof}

In light of the definition of a Burch submodule of some module, the following generalizes \cite[Corollary 1(ii), page 947]{b}.  

\begin{cor}\label{bpd}
Let $R$ be a singular local ring. Let $N$ be an $R$-module. If either $\pd N<\infty$, or $\id N<\infty$, then $N$ cannot be isomorphic to a Burch submodule of any $R$-module.   
\end{cor}  

\begin{proof} $\pd N<\infty$ (resp. $\id N<\infty$) implies $\Tor^R_l(k,N)=0$ (resp. $\Ext^l_R(k,N)=0$) for all $l\gg 0$. If $N$ were isomorphic to a Burch submodule of some $R$-module, then Proposition \ref{finpd3} would imply that $\pd k <\infty$, so $R$ is regular, contradicting $R$ is singular. This proves the claim.  
\end{proof}    

For our next result, we record the following easy observation.

\begin{lem}\label{ten}
Let $M$ be an $R$-module and $I$ an ideal of $R$.
If $\Tor^R_1(M,R/I)=0$, then $I\otimes_R M\cong IM$.
\end{lem}

\begin{proof}
Tensoring the exact sequence $0\to I\to R \to R/I \to 0$ with $M$ we get the exact sequence $0=\Tor^R_1(M,R/I)\to I\otimes_R M\to M\to M/IM\to 0$ which gives the required isomorphism.  
\end{proof}

The following generalizes part of \cite[2.2]{cors}.

\begin{prop} 
Let $R$  be a singular local ring. Let $X$ be an $R$-module. Let $x_1,\cdots,x_n$ be an $R$ and $X$-regular sequence. Assume $\pd_R X<\infty$. Put $L=\left((x_1,\cdots,x_n)X:_X \m\right)$. Then, it holds that $\m L=\m(x_1,\cdots,x_n)X$, and that $\mu(L)=n\cdot\mu(X)+\ell_R \Ext^n_R(k,X)$. 
\end{prop}  

\begin{proof}
To show $\m\left((x_1,\cdots,x_n)X:_X \m\right)=\m(x_1,\cdots,x_n)X$ (i.e., that $(x_1,...,x_n)X$ is not Burch in $X$), by Corollary \ref{bpd}, it is enough to show $\pd_R  (x_1,\cdots,x_n)X<\infty$. Since $x_1,\cdots,x_n$ is $R$ and $X$-regular sequence, so $\Tor^R_{>0}(R/(x_1,\cdots,x_n),X)=0$, thus $\pd_R(X/(x_1,\cdots,x_n)X)=\pd_R(X\otimes_R R/(x_1,\cdots,x_n))=\pd_R(X)+\pd_R(R/(x_1,\cdots,x_n))<\infty$.  So along with $\pd_R(X)<\infty$, we now get $\pd_R (x_1,\cdots,x_n)X<\infty$. Thus we get $\m L=\m(x_1,\cdots,x_n)X$. This implies $\mu(L)=\mu((x_1,...,x_n)X)+\ell\left(\dfrac{L}{(x_1,...,x_n)X}\right)$. By Lemma \ref{ten}, $\mu\left((x_1,\cdots,x_n)X \right)=\mu((x_1,\cdots,x_n)\otimes_R X)=\mu\left((x_1,\cdots,x_n)\right)\cdot \mu(X)=n\cdot\mu(X)$. Moreover, $L/(x_1,...,x_n)X=\soc(X/(x_1,...,x_n)X)\cong \Hom_{R/(x_1,...,x_n)}(k,X/(x_1,...,x_n)X)\cong \Ext^n_R(k,X)$, where the last isomorphism holds since $x_1,\cdots,x_n$ is $R$ and $X$-regular. This completes the proof.   
\end{proof}    

\section{Weakly $\m$-full submodules}   

In this section, we naturally extend the Definition of weakly $\m$-full ideals \cite[Definition 3.7]{two} to submodules of a module in general.

\begin{dfn} Let $M$ be an $R$-module. A submodule $N$ of $M$ is called \textit{weakly $\m$-full} if $(\m N:_M \m)=N$.  
\end{dfn}

We note that weakly $\m$-full submodules of $R$ are exactly the weakly $\m$-full ideals in the sense of \cite[Definition 3.7]{two}. For $\m$-primary submodules, these are the same as Basically full submodules (\cite[Theorem 2.12]{basic}).  

\begin{rem}\label{subre} Note that if $N\subseteq Y\subseteq X$ are submodules, then $N\subseteq (\m N:_Y \m)\subseteq (\m N:_X \m)$. Thus, if $N$ is weakly $\m$-full in $X$, then $N$ is weakly $\m$-full in $Y$.
\end{rem}

The following generalizes \cite[3.11]{two}, see also \cite[Corollary 2.4]{burch}.

\begin{lem}\label{wb}
Let $N$ be weakly $\m$-full submodule of an $R$-module $M$.
If $\depth (M/N)=0$ (e.g. $M/N$ has finite length), then $N$ is Burch. In particular, if $0$ is a weakly $\m$-full submodule of an $R$-module $M$, then $\depth M>0$. 
\end{lem}    

\begin{proof}  If $N$ is not Burch, then $\m(N:_M\m)=\m N$. Then, $(N:_M\m)\subseteq (\m N:_M \m)=N$, where the last equality holds because $N$ is weakly $\m$-full.
Thus, $N=(N:_M \m)$.
Hence $\depth(M/N)>0$, a contradiction.
Thus $N$ is Burch. In particular, if $0$ is weakly $\m$-full submodule of an $R$-module $M$ and we have $\depth M=0$, then $0$ would be Burch in $M$, contradicting the Remark made in Definition \ref{burchdef}; so this proves the last part of the claim.  
\end{proof} 

Note that if $N$ is a submodule of an $R$-module $M$ such that $\depth(M/N)>0$, i.e., $(N:_M\m)=N$, then obviously $(\m N:_M\m)=N$, so $N$ is a weakly $\m$-full submodule of $M$. Hence the condition $\depth(M/N)=0$ cannot be dropped from  Lemma \ref{wb} in view of Lemma \ref{dep}.  

We recall the definition of $\m$-full submodules.

\begin{dfn}(\cite[2.1]{pu}) A submodule $N$ of $M$ is called $\m$-full if $(\m N:_M x)=N$ for some $x\in \m$.   
\end{dfn}  
 
We have a strata of classes of submodules
\[
\text{$\m$-full} \implies \text{weakly $\m$-full} \overset{\text{ when } \depth (M/N)=0} \implies \text{Burch},
\]
where the first implication is by definition, and the second one is due to Lemma \ref{wb}.

 The following two lemmas give other large class of examples of (weakly) $\m$-full submodules. 

\begin{lem} \label{l}
Let $N$ be a submodule of an $R$-module $X$.
Then $N:_X \m$ is a weakly $\m$-full submodule of $X$.
\end{lem}

\begin{proof}
Set $N':=N:_X \m$.
Then $\m N'\subseteq N$.
Thus one has $N'\subseteq \m N':_X\m \subseteq N:_X\m=N'$, which shows that $N'=\m N':_X \m$.
\end{proof}

\begin{lem}\label{highpower}
Let $R$ be a local ring with infinite residue field.
Let $N$ be a submodule of an $R$-module $M$.
If $\depth M>0$, then $\m^n N$ is an $\m$-full submodule of $M$ for all $n\gg 0$.   
\end{lem} 

\begin{proof} 
By \cite[Lemma 2.1]{da} (or \cite[Lemma 8.5.3, Proposition 8.5.7.(2)]{hs} when $M=N$) we have that for some $x\in \m$, $\m^nN:_Mx=\m^{n-1}N$ for all $n\gg 0$.
Hence, $\m^{n-1}N$ is $\m$-full for all $n\gg 0$. 
\end{proof}  

In view of Lemma \ref{burchsum}, the following observation generalizes \cite[Lemma 2.1]{GH} which was the key ingredient for proving \cite[Theorem 1.1]{GH}.    

\begin{prop}
Let $N\subseteq M$ be $R$-submodules of an $R$-module $X$ such that $(N:_X \m)\nsubseteq M$.
If $M$ is a weakly $\m$-full submodule of $X$, then $N$ is Burch in $X$.  
\end{prop}

\begin{proof} If $N$ is not Burch in $X$, then $(N:_X \m)=(\m N:_X \m)$ (Lemma \ref{burchalter}), and then $(N:_X \m)\subseteq (\m M:_X \m)=M$, where the latter equality holds since $M$ is weakly $\m$-full in $X$.
This contradicts $(N:_X \m)\nsubseteq M$. Thus, $N$ is Burch in $X$.  
\end{proof}

Now we proceed to investigate rigid test property of weakly $\m$-full submodules. Motivated by \cite{ot}, we prove the next two results in the greater generality of the condition $(N:_X J)=(\m N:_X \m J)$, so that specializing to $J=R$, we can conclude rigid test property of weakly $\m$-full submodules.  Note that since $(\m N:_X \m J)=(\m N:_X \m):_X J$, so if $N$ is weakly $\m$-full submodule of $X$, then it satisfies $(N:_X J)=(\m N:_X \m J)$ for all $J$.  

\begin{lem} \label{ll}
Let $M$ and $X$ be $R$-modules, $J$ be an ideal of $R$, and $N$ be a submodule of $X$ such that $N:_X J=\m N:_X \m J$.
Also let $N'$ be a submodule of $X$ such that $N:_XJ\subseteq N'\subseteq (N:_X \m J)$.
Assume there exists an integer $t\ge 1$ such that $J\Ker(\partial_t\otimes (N:_XJ))\subseteq \im(\partial_{t+1}\otimes N)$, where $(F_n,\partial_n)$ is a minimal free resolution of $M$.
Then $\Ker(\partial_t \otimes N')=\Ker(\partial_t\otimes (N:_XJ))$ (and hence $J\Ker(\partial_t \otimes N')\subseteq \im(\partial_{t+1}\otimes N)\subseteq \im(\partial_{t+1}\otimes (N:_X\m))$). 
\end{lem}

\begin{proof}
Set $L=N:_XJ$.
Note that $\m N'\subseteq L\subseteq N'$.
Then we have a commutative diagram
\[
\begin{tikzcd}
0 \ar[r] & F_t\otimes_R L \ar[r] \ar[d, "\partial_t\otimes L"] & F_t\otimes_R N' \ar[r] \ar[d,"\partial_t\otimes N'"] & F_t\otimes_R (N'/L) \ar[r] \ar[d,"\partial_t\otimes (N'/L)"] & 0 \\
0 \ar[r] & F_{t-1}\otimes_R L \ar[r] & F_{t-1}\otimes_R N' \ar[r] & F_{t-1}\otimes_R (N'/L) \ar[r] & 0
\end{tikzcd}
\]
with exact rows.
It induces an exact sequence $0 \to \Ker(\partial_t\otimes L) \xrightarrow[]{\alpha} \Ker(\partial_t\otimes N') \xrightarrow[]{\beta} \Ker(\partial_t\otimes (N'/L))$.
Suppose that $\beta$ is nonzero.
Take an element $x\in \Ker(\partial_t\otimes N')$ such that $\beta(x)\not=0$.
Since $N'/L$ is annihilated by $\m$ and $\partial_t(F_t)\subseteq \m F_{t-1}$, $\partial_t\otimes (N'/L)$ is a zero map, i.e., $\Ker(\partial_t\otimes (N'/L))=F_t\otimes_R (N'/L)$.
Also, by commutativity of the diagram, $\beta\colon\Ker(\partial_t\otimes N')\to \Ker(\partial_t\otimes (N'/L))=F_t\otimes_R (N'/L)$ is just the restriction of the canonical surjection $F_t\otimes_R N' \to F_t \otimes_R (N'/L)$.  
So, if $x\in F_t \otimes_R L$, we get $\beta(x)=0$ in $F_t \otimes_R (N'/L)$.
Thus it yields that $x\in F_t\otimes_R N' \smallsetminus F_t \otimes_R L$.
Since $\m x\in F_t\otimes_R L$, we have
\begin{align*}
J\m x & \subseteq J\left(\Ker(\partial_t\otimes N')\cap (F_t\otimes_R L)\right)\\
& = J\Ker(\partial_t\otimes L)\\
& \subseteq \im(\partial_{t+1} \otimes N)\\
& \subseteq F_t\otimes_R \m N.  
\end{align*}
Here at the third line we use the assumption $J\Ker(\partial_t\otimes (N:_XJ))\subseteq \im(\partial_{t+1}\otimes N)$.  
It follows that $x\in F_t\otimes_R (\m N:_X \m J)=F_t\otimes_R L$, a contradiction.
Thus we see that $\beta$ is a zero map, i.e., $\alpha$ is an isomorphism, i.e., $\Ker(\partial_t \otimes L)=\Ker(\partial_t\otimes N')$.
The remaining assertion is just a direct calculation.
\end{proof}  

\begin{lem} \label{pp}
Let $M$ and $X$ be $R$-modules, $J$ be an ideal of $R$ and $N$ be a submodule of $X$ such that $N:_XJ=\m N:_X\m J$.
Assume $X/(N:_XJ)$ has finite length and there exists an integer $t\ge 1$ such that $J\Ker(\partial_t\otimes (N:_XJ))\subseteq \im(\partial_{t+1}\otimes N)$, where $(F_n,\partial_n)$ is a minimal free resolution of $M$.  
Then $\Ker(\partial_t\otimes (N:_XJ))=\Ker(\partial_t\otimes X)$.
\end{lem}

\begin{proof} Set $L=N:_XJ$. Since $X/L$ has finite length, we have $\m^l X \subseteq L$ for some $l$.
Thus we get $X=(\cdots (L:_X\m):_X \m)\cdots):_X \m)$, where the colon operation is taken $l$-times.
By Lemma \ref{ll}, we get $\Ker(\partial_t\otimes L)=\Ker(\partial_t\otimes (N:_X\m J))=\Ker(\partial_t\otimes (L:_X\m))$ and $J\Ker(\partial_t\otimes L)\subseteq \im(\partial_{t+1}\otimes N)\subseteq \im(\partial_{t+1}\otimes (N:_X\m))$.  
Observe inclusions
\[
(N:_X\m):_XJ\subseteq \left(\m(N:_X\m)\right):_X\m J\subseteq N:_X\m J=(N:_X\m):_XJ
\]
and so an equality $(N:_X\m):_XJ=\left(\m(N:_X\m)\right):_X\m J$.
Since $L:_X\m=(N:_X\m):_XJ$, we may replace $N$ with $N:_X\m$ and obtain $\Ker(\partial_t\otimes ((L:_X\m):_X\m))=\Ker(\partial_t\otimes (L:_X\m))$ and 

$J\ker(\partial_t\otimes((L:_X\m):_X\m))\subseteq \im(\partial_{t+1}\otimes (N:_X\m))\subseteq \im(\partial_{t+1}((N:_X\m):_X\m))$ by applying Lemma \ref{ll} again. 
Continuing this way, we get $\Ker(\partial_t\otimes (N:_XJ))=\Ker(\partial_t\otimes X)$.
\end{proof}

\begin{thm} \label{t412}
    Let $M$ and $X$ be $R$-modules, $J$ be an ideal of $R$, and $N$ be a submodule of $X$ such that $N\subseteq \m J X$, $X/(N:_XJ)$ has finite length, and $N:_XJ=\m N:_X\m J$.
Assume there exists an integer $t\ge 0$ such that $\Tor_t^R(M,N)=0$.
Then the following holds:
\begin{enumerate}[\rm(1)]
\item For a minimal free resolution $(F_n,\partial_n)$ of $M$, $\partial_{t+1}\otimes(JX)=0$, i.e., $I_1(\partial_{t+1})\subseteq \ann_R(JX)$.
\item If $X$ is faithful and $J$ contains a non zero-divisor, then $\pd M\le t$.
\item  If $\supp(JX)=\spec(R)$ and $M$ is locally free on $\Ass(R)$, then $\pd M\le t$.
\end{enumerate} 
\end{thm}

\begin{proof} For $t=0$, we get $M\otimes_R N=0$, so either $M=0$, or $N=0$. If $N=0$, then $0:_XJ=0:_X\m J=(0:_X J):_X\m$, so $\soc(X/(0:_X J))=0$. Therefore, $X/(0:_X J)$ having finite length implies $X=(0:_X J)$, so $JX=0$, and we are done. If $M=0$, then we are obviously done. For the rest of the proof, we concentrate on $t\ge 1$.   

(1): We may assume $JX\not=0$; in the case $JX=0$, the assertion is trivial.
Set $L:=N:_XJ$.
Observe the following inclusions
\begin{align*}
J\Ker(\partial_t\otimes L)\subseteq \Ker(\partial_t\otimes L)\cap F_t\otimes_R JL&\subseteq \Ker(\partial_t\otimes L)\cap F_t\otimes_R N\\ &=\Ker(\partial_t\otimes N)= \im(\partial_{t+1}\otimes N).
\end{align*}
Here the second equality follows by the assumption $\Tor_t^R(M,N)=0$.
By Lemma \ref{pp}, we obtain that $\Ker(\partial_t\otimes L)=\Ker(\partial_t\otimes X)$.
Then one gets the following calculations
\begin{align*}
J\im(\partial_{t+1}\otimes X)\subseteq J\Ker(\partial_t\otimes X)=J\Ker(\partial_t\otimes L) &\subseteq \Ker(\partial_t\otimes L)\cap F_t\otimes_R JL\\
&\subseteq \Ker(\partial_t\otimes L)\cap F_t\otimes_R N
=\Ker(\partial_t\otimes N)\\
&=\im(\partial_{t+1}\otimes N)\\
&\subseteq \im(\partial_{t+1}\otimes \m JX)= \m J\im(\partial_{t+1}\otimes X).
\end{align*}
Here the third equality follows by the assumption $\Tor_t^R(M,N)=0$.
By Nakayama's lemma, $J\im(\partial_{t+1}\otimes X)$ must be zero, that is, $\partial_{t+1}\otimes (JX)=0$.

(2) Assume $X$ is faithful and $J$ contains a non zero-divisor.
Then $\ann_R(JX)=\ann_R(X):_RJ=0:_RJ=0$.
Therefore, by (1), we see that $\partial_{t+1}=0$, i.e., $\pd M\le t$.

(3) Follows by (1) and Lemma \ref{ll27}(3).
\end{proof}   

Taking $J=R$ in Theorem \ref{t412} and remembering $(N:_XR)=N$, we conclude the following test-rigidity property of weakly $\m$-full submodules. 

\begin{cor}\label{wfinpd}  
Let $M$ and $X$ be $R$-modules, and $N$ be a  weakly $\m$-full submodule of $X$.
Assume $N\subseteq \m X$, $X/N$ has finite length, and there exists an integer $t\ge 0$ such that $\Tor_t^R(M,N)=0$.  
Then, the following holds:
\begin{enumerate}[\rm(1)]
\item For a minimal free resolution $(F_n,\partial_n)$ of $M$, $I_1(\partial_{t+1})\subseteq \ann(X)$.
\item If $X$ is faithful, then $\pd M \le t$.
\item  If $\supp(X)=\spec(R)$ and $M$ is locally free on $\Ass(R)$, then $\pd M\le t$.  
\end{enumerate} 
\end{cor}

The result of Celikbas and Kobayashi \cite[Theorem 2.8]{ot} is the special case $X=R$ of the following Theorem \ref{Jfull}(3).   

\begin{thm}\label{Jfull} Let $M$ and $X$ be $R$-modules and $J$ be an ideal of $R$.
Let $N$ be a submodule of $X$ such that $N\subseteq \m J X$,  $X/(N:_X J)$ has finite length and $(N:_X J)=(\m N:_X \m J)$.
Assume there exists an integer $t\ge 1$ such that $\Tor_t^R(M,X/N)=0$. Then, the following holds:   
\begin{enumerate}[\rm(1)]
\item For a minimal free resolution $(F_n,\partial_n)$ of $M$, $\partial_t \otimes (JX)=0$, i.e., $I_1(\partial_t)\subseteq \ann_R(JX)$.  
\item If $X$ is faithful and $J$ contains a non zero-divisor, then $\pd M<t$. 
\item  If $\supp(JX)=\spec(R)$ and $M$ is locally free on $\Ass(R)$, then $\pd M< t$.
\item If $\Tor_t^R(M,X)=0$, then $J((\syz^t M)\otimes_R X)=0$. 
\end{enumerate}  
\end{thm} 

\begin{proof} (1):We may assume $JX\not=0$; in the case $JX=0$, the assertion is trivial.
First we prove $J\im(\partial_t \otimes (X/N))=0$, or equivalently, $J\im(\partial_t \otimes X)\subseteq F_{t-1}\otimes_R N$.
If this is not true, then $\im(\partial_t \otimes X)\nsubseteq F_{t-1}\otimes_R (N:_X J)$, so the image $\overline{\im(\partial_t \otimes X)}$ of $\im(\partial_t \otimes X)$ in $(F_{t-1}\otimes_R X) /\left(F_{t-1}\otimes_R (N:_X J)\right)$ is non-zero.
Since $(F_{t-1}\otimes_R X)/(F_{t-1}\otimes_R (N:_X J))\cong F_{t-1}\otimes_R \left(X/(N:_X J)\right)$ has finite length, and socle of a finite length module is essential submodule, so $\soc \left((F_{t-1}\otimes_R X)/(F_{t-1}\otimes_R (N:_X J))\right)\cap  \overline{\im(\partial_t \otimes X)}\ne 0$ (Remark that $X/(N:_X J)\not=0$ as $JX\not\subseteq N$ by the assumption $N\subseteq \m J X$ and $JX\ne 0$ with the Nakayama's lemma).  
Take a preimage $a\in \im(\partial_t \otimes X)\subseteq F_{t-1}\otimes X$ 
of a nonzero element $\bar{a}$ of $\soc \left((F_{t-1}\otimes_R X)/(F_{t-1}\otimes_R (N:_X J))\right)\cap  \overline{\im(\partial_t \otimes X)}$.
By the choice of $a$, $a\notin F_{t-1}\otimes_R (N:_X J)$ and $\m a \in F_{t-1}\otimes_R (N:_X J)$.
So, $\m J a \in F_{t-1}\otimes_R N$, that is, $\m J\bar a=0$ in $(F_{t-1}\otimes_R X)/(F_{t-1}\otimes_R N)=F_{t-1}\otimes_R (X/N)$.
Now, $a$ belongs to $\im(\partial_t \otimes X)$, which means that $a=(\partial_t \otimes X)(b)$ for some $b\in F_t\otimes_R X$.
Then the image $\bar{b}$ of $b$ in $F_t\otimes_R (X/N)$ satisfies $(\partial_t\otimes(X/N))(\m J\bar b)=\m J\bar a=0$.
By using the assumption $\Tor_t^R(M,X/N)=0$, get $\m J\bar b\subseteq \Ker(\partial_t\otimes(X/N))=\im(\partial_{t+1}\otimes(X/N))$. Lifting the inclusion, one has $\m J b\subseteq \im(\partial_{t+1}\otimes X)+F_t\otimes_R N$, and so $\m J (\partial_t\otimes X)(b)=(\partial_1\otimes X)(\m Jb)\subseteq (\partial_t\otimes X)(F_t\otimes_RN)\subseteq \m F_{t-1}\otimes_R N= F_{t-1}\otimes_R \m N$.
Thus $a=(\partial_t\otimes X)(b)\in F_{t-1}\otimes_R(\m N:_X \m J)=F_{t-1}\otimes_R (N:_X J)$, contradicting our choice of $a$.
Thus, we must have $J\im(\partial_t \otimes (X/N))=0$.
It means that $J(F_t\otimes_R(X/N))\in \Ker(\partial_t\otimes(X/N))=\im(\partial_{t+1}\otimes(X/N))$, and hence $J(F_t\otimes_R X)\subseteq \im (\partial_{t+1}\otimes X)+F_t\otimes_R N$. Since $N\subseteq \m JX\subseteq JX$, We get
\begin{align*}
J(F_t\otimes_R X) &=\left(\im (\partial_{t+1}\otimes X)+F_t\otimes_R N\right)\cap J(F_t\otimes_R X)\\
&=\left(\im (\partial_{t+1}\otimes X)\right)\cap J(F_t\otimes_R X))+F_t\otimes_R N\\
&\subseteq\left(\im (\partial_{t+1}\otimes X)\right)\cap J(F_t\otimes_R X))+F_t\otimes_R \m JX\\
&\subseteq F_t\otimes_R JX. 
\end{align*}  
Here the second equality follows by modular low.
Thus, $J(F_t\otimes_R X)=\left(\im (\partial_{t+1}\otimes X)\cap J(F_t\otimes_R X)\right)+ \m J(F_t\otimes_R X)$.
By Nakayama's lemma, we have $J(F_t\otimes_R X)=\im (\partial_{t+1}\otimes X)\cap J(F_t\otimes_R X)$.
So, $J(F_t\otimes_R X)\subseteq \im (\partial_{t+1}\otimes X)\subseteq \Ker (\partial_{t}\otimes X)$, which shows $J(\partial_t \otimes X)=0$. 

(2) We have $\ann_R(JX)=(\ann_R(X):_R J)$, so if $X$ is faithful, then $\ann_R(JX)=(0:_R J)$, and the latter is $0$ if $J$ contains a non zero-divisor.
So we get $\partial_t=0$, and so $\pd M<t$. 

(3) Follows by (1) and Lemma \ref{ll27}(3).

(4) If $\Tor_t^R(M,X)=0$, then $\im(\partial_t\otimes X)\cong (\syz^t M) \otimes_R X$ as in the proof of Proposition \ref{faithquot}(6), hence the claim follows.     
\end{proof}

\begin{cor} Let $M$ and $X$ be $R$-modules and $J$ be an ideal of $R$.
Let $N$ be a submodule of $X$ such that $N\subseteq \m J X$,  $X/(N:_X J)$ has finite length and $(N:_X J)=(\m N:_X \m J)$. Assume there exists an integer $t\ge 1$ such that either $\Ext^t_R(M,N)=0$ or $\Ext^t_R(M,X/N)=0$. Also assume that one of the following holds: \begin{enumerate}[\rm(1)]

\item If $X$ is faithful and $J$ contains a non-zero divisor. 

\item $\supp(JX)=\spec(R)$ and $M$ is locally free on $\Ass(R)$.
\end{enumerate}

Then, $\pd M\le t$.  
\end{cor}

\begin{proof} By \cite[4.1(i)]{two}, $\Ext^t_R(M,N)=0$  (or $\Ext^t_R(M,X/N)=0$) implies $\Tor^R_1(\Tr\syz^tM,N)=0$  (or $\Tor^R_1(\Tr\syz^tM,X/N)=0$ respectively), where $\Tr(-)$ denotes Auslander transpose. Note that if $M$ is locally free on $\Ass(R)$, then so is $\Tr\syz^t M$. By Theorem \ref{t412} and  \ref{Jfull} (2) and (3), it follows that $\pd \Tr \syz^t M \le 1$.  As $t\ge 1$, we also have $\Ext^1_R(\Tr \syz^t M,R)=0$ by \cite[4.1(ii)]{two}. Hence, $\Tr \syz^t M$ is free. Thus, $\syz^t M$ is free, i.e., $\pd M\le t$. 
\end{proof}   

Taking $J=R$ in Theorem \ref{Jfull} and remembering $(N:_XR)=N$, we immediately obtain the corollary below.  

\begin{cor}\label{th411}
Let $M$ and $X$ be $R$-modules, and $N$ be a weakly $\m$-full submodule of $X$.
Assume $N\subseteq \m X$, $X/N$ has finite length, and there exists an integer $t>0$ such that $\Tor_t^R(M,X/N)=0$.
Then, the following holds:
\begin{enumerate}[\rm(1)]
\item For a minimal free resolution $(F_n,\partial_n)$ of $M$, $I_1(\partial_t)\subseteq \ann(X)$.

\item If $X$ is faithful, then $\pd M<t$.
\item  If $\supp(X)=\spec(R)$ and $M$ is locally free on $\Ass(R)$, then $\pd M< t$.
\item If $\Tor^R_t(M,X)=0$ and $X\not=0$, then $\pd M<t$.  
\end{enumerate}  
\end{cor}

As a consequence, we can recover the result of Celikbas, Goto, Takahashi and Taniguchi \cite{hw}.

\begin{cor}[Celikbas--Goto--Takahashi--Taniguchi]
Let $I$ be an $\m$-primary weakly $\m$-full ideal of $R$, and $M$ be an $R$-module.
Assume there exists an integer $t\ge 0$ such that $\Tor_t^R(M,R/I)=0$.
Then $\pd M<t$.  
\end{cor}

\section{Vanishing of (co)homology over Cohen--Macaulay local rings}  

Let $(R,\m)$ be a local ring and $M$ is a Cohen--Macaulay $R$-module. Recall that $M$ is called \textit{Ulrich} if $e(M)=\mu(M)$, where $e(M)$ is the multiplicity of $M$ with respect to the ideal $\m$ \cite[Definition 2.1]{agl}. Note that when $R$ is moreover Cohen--Macaulay and $M$ is maximal Cohen--Macaulay, this coincides with the definition of Ulrich modules as introduced in \cite{bhu}.   

\begin{lem}\label{tenfaith} Let $R\to S$ be a flat extension of  rings. Let $M$ be an $R$-module and $I$ an ideal of $R$. Then, $S \otimes_R (IM)\cong (IS)(S \otimes_R M)$.  
\end{lem}

\begin{proof} Consider the exact sequence $0\to IM \to M \to M/IM\to 0$ which after tensoring with $S$ gives $0\to S \otimes_R (IM) \to S\otimes_R M \to S\otimes_R M/IM \to 0$. Now $S \otimes_R M/IM \cong S \otimes_R (M\otimes_R R/I)\cong (S\otimes_R M)\otimes_R R/I\cong (S\otimes_R M)/I(S \otimes_R M)$. Now by the natural $S$-module structure on $S\otimes_R M$, we see that $I(S \otimes_R M)=(IS)(S \otimes_R M)$. Thus we get $0\to S \otimes_R (IM) \to S\otimes_R M \to (S \otimes_R M)/(IS)(S \otimes_R M)\to 0$, and by naturality of the isomorphisms, we see that the map $S\otimes_R M \to (S \otimes_R M)/(IS)(S \otimes_R M)$ in the exact sequence has kernel $(IS)(S \otimes_R M)$, hence $S \otimes_R (IM)\cong (IS)(S \otimes_R M).$  
\end{proof}

In one--dimensional case, we see the following characterization of Ulrich modules.
\begin{lem} \label{l51}
Let  $M$ be a Cohen--Macaulay $R$-module of dimension one.
Then $M$ is Ulrich if and only if $M\cong \m M$.
\end{lem}  

\begin{proof}
Consider the faithfully flat extension $S=R[X]_{\m[X]}$. Then, due to Lemma \ref{tenfaith} and \cite[Proposition 2.5.8]{Gro}, $M\cong \m M$ if and only if $S\otimes_R M \cong (\m S)(S\otimes_R M)$. Moreover, $S\otimes_R M$ is Ulrich $S$-module if and only if $M$ is Ulrich $R$-module (\cite[Proposition 2.2(3)]{agl}). Thus, passing to $S$, we can assume that the residue field is infinite.  

If $M$ is Ulrich, then since $\dim M=1$, so $\m M=fM$ for some $f\in \m$ by \cite[Proposition 2.2(2)]{agl}. Then, $f$ is a system of parameters for $M$, hence $f$ is $M$-regular by \cite[theorem 2.1.2(d)]{bh1}. Hence, $M\cong f M=\m M$. 

Now assume $M\cong \m M$.
Then for each $i\ge 0$, $M\cong \m^i M$, in particular, $\mu(M)=\mu(\m^i M)$.
Thus we have
\[
\ell(M/\m^n M)=\sum_{i=0}^{n-1}\ell(\m^i M/\m^{i+1} M)=\sum_{i=0}^{n-1}\mu(\m^iM)=n\cdot \mu(M).
\]
Since $\dim M=1$, so it follows that $e(M)=\lim\nolimits_{n\to \infty} \frac{1}{n}\ell(M/\m^n M)=\mu(M)$, that is, $M$ is Ulrich.  
\end{proof}

\begin{lem}\label{Ulrichfull} Let $R$ be a local ring  with an infinite residue field. Let $M$ be an Ulrich $R$-module of dimension one. Then, $M$ has a weakly $\m$-full submodule $N$ such that $N\subseteq \m M$, $M/N$ has finite length and $M\cong N$. 
\end{lem}  

\begin{proof} We have $\m M=xM$ for some $M$-regular element $x\in \m$ (\cite[Proposition 2.2(2)]{agl}). So, $\m^n M=x^n M, \forall n\ge 1$. Since $\depth M=1>0$, so by Lemma \ref{highpower}, $\m^lM=x^lM$ is a weakly $\m$-full submodule of $M$ for some $l>1$. Put $N:=\m^l M$. Then $N\subseteq \m M$, $N=x^lM\cong M$ and $M/N$ has finite length.  
\end{proof}   

The lemma below gives a variant of \cite[3.1]{gp}.

\begin{lem} \label{l52}
Let  $M,N$ be $R$-modules, where $N$ is Ulrich, and has dimension $1$. Let $t\ge 0$ be an integer. 
Assume one of the following conditions hold:

\begin{enumerate}[\rm(1)]
\item $N$ is faithful.
    
\item $\supp(N)=\spec(R)$ and $M$ is locally free on $\Ass(R)$. 
\end{enumerate}

If either $\Ext^{t+1}_R(M,N)=0$, or $\Tor^R_t(M,N)=0$ holds, then $\pd M\le \min\{t,1\}$.  
\end{lem}  

\begin{proof}  
Since $N$ is Ulrich, and has dimension $1$, so $N\cong \m N$ by Lemma \ref{l51}.
Thus $\Tor^R_t(M,\m N)$ (resp. $\Ext^{t+1}_R(M,\m N)$) vanishes.
Thanks to Proposition \ref{t}(4),(5) and Proposition \ref{tt}(4),(5) we see an inequality $\pd M\le t$. In either of the case (1) or (2), $N$ has full support, so $1=\dim N=\dim R\ge \depth R$. Hence, by the Auslander-Buchsbaum formula, it follows that $\pd M\le  1$. 
\end{proof}    

\begin{rem}\label{ulrem}  
The Tor vanishing part of Lemma \ref{l52} also follow from Lemma \ref{Ulrichfull} and Corollary \ref{wfinpd}(2) and (3).   
\end{rem}


\begin{rem}
Let $R$ be a Cohen--Macaulay local ring with a canonical module $\omega$.
If $M$ is a maximal Cohen--Macaulay Ulrich $R$-module then so is $M^\dag$ \cite[4.1]{kt}. 
\end{rem}

The result below is proven by Auslander and Buchweitz \cite{ab}, which allows us to replace arbitrary modules by maximal Cohen--Macaulay
modules when dealing with vanishing of some Ext modules; see the
subsequent remark.

\begin{thm}[Auslander--Buchweitz] \label{mcmapp}
Let $R$ be a Cohen--Macaulay local ring with a canonical module $\omega$ and
let $N$ be an $R$-module.
Then there exist $R$-modules $Y$ and $L$, such that $L$ is maximal Cohen--Macaulay, $Y$ has
finite injective dimension, and they form a short exact sequence $0 \to Y \to L \to N \to 0$.
\end{thm}

\begin{rem} \label{mcmapp1}
Let $R$ be a Cohen--Macaulay local ring with a canonical module $\omega$ and
let $M$ and $N$ be $R$-modules.
Assume $M$ is maximal Cohen--Macaulay.
Take a short exact sequence $0 \to Y \to L \to N \to 0$ as in the previous theorem.
Since $\id Y<\infty$, $\Ext^i_R(M,Y)=0$ for any $i>0$.
Thus for any $i>0$ the vanishing of $\Ext^i_R(M,N)$ implies that of $\Ext^i_R(M,L)$. 
\end{rem}

\begin{rem}\label{faithcm} When $N$ is a faithful Ulrich $R$-module, then it follows that $N$ is maximal Cohen--Macaulay. Indeed, $N$ is faithful and Cohen--Macaulay implies $\dim R=\dim N=\depth N$. 
\end{rem} 

When testing against faithful Ulrich modules, the following theorem shows that the $(d+1)$-consecutive vanishing as given in \cite[Theorem 3.1, Corollary 3.4]{gp} can be improved to $d$-many consecutive vanishing. For consideration regarding number of $\Tor$ and $\Ext$ vanishing against not necessarily maximal Cohen--Macaulay Ulrich modules, the reader may also see \cite[Proposition 2.5]{dg}.   

\begin{thm}\label{faithulhigh}
Let $R$ be a local Cohen--Macaulay ring of dimension $d\ge 1$.  Let $M,N$ be $R$-modules such that $N$ is Ulrich and faithful. Then the following hold:  

\begin{enumerate}[\rm(1)]
\item 
If there exists an integer $n\ge 1$ such that $\Tor^R_i(M,N)=0$ for all $n\le i\le n+d-1$, then $\pd_R M<\infty$.

\item Assume $\depth M\ge d-1$. If there exists an integer $n\ge 1$ such that $\Ext^i_R(M,N)=0$ for $n\le i\le n+d-1$, then $\pd_R M<\infty$.

\item Assume $\depth M\le d-1$. If there exists an integer $n\ge 0$ such that $\Ext^i_R(M,N)=0$ for $n+d-\depth M \le i\le n+2d-\depth M-1$, then $\pd_R M<\infty$.

\item If there exists an integer $n\ge 1$ such that $\Ext^i_R(N,M)=0$ for $n\le i\le n+d-1$, then $\id_R M<\infty$.

\end{enumerate}

\end{thm}  
\begin{proof}

We may pass to the faithfully flat extension $S:=R[X]_{\m[X]}$ and assume the residue field is infinite and then again to the faithfully flat extension of completion to assume $R$ has a canonical module (the property of being a faithful module is preserved under passing to flat extension by Remark \ref{fflat}).  

We note that as $N$ is maximal Cohen--Macaulay by Remark \ref{faithcm}, hence $\dim N=d$.

(1)  We prove the claim by induction on $d$.
Lemma \ref{l52}(1) deals with the base case $d=1$. Suppose that $d\ge 2$. 
Since $\m$ is not contained in any associated prime of $R$, so by \cite[8.5.9]{hs}, choose a superficial element  $x\in \m$ for $\m$ with respect to $N$ and which does not belong to $\bigcup_{\p\in \Ass(R)\cup \Min(\sigma^R(N))\smallsetminus \{\m \}}\p$, so that $x$ is $R$-regular, so also $N$-regular, and $N/xN$ is Ulrich over $R/xR$ (\cite[Proposition 11.1.9]{hs}). By Lemma \ref{fs}, $N/xN$ is faithful. Applying $M\otimes_R(-)$ to the exact sequence $0\to N \xrightarrow{x} N \to N/xN\to 0$ gives us vanishings of $\Tor^R_{i}(M,N/x N)$ for $i=n+1,\dots,n+d-1$.
Hence, $\Tor^R_i(M',N/xN)=0$ for $i=n,\dots,n+d-2$, where $M':=\syz_R M$.
Then, $x$ is $M'$-regular, hence \cite[Lemma 2, page 140]{mat} gives us $\Tor^{R/xR}_i(M'/xM',N/xN)=0$ for $i=n,\dots,n+d-2$.
Since $\dim R/xR=d-1\ge 1$, so now the induction hypothesis yields that $\pd_{R/xR} M'/xM'<\infty$, thus $\pd_R M'<\infty$, so $\pd_R M<\infty$.    

(2) Again we go by induction on $d$.
The base case $d=1$ follows by Lemma \ref{l52}(1) as $N$ is faithful.
Now suppose that $d\ge 2$ and $\Ext^i_R(M,N)=0$ for $n\le i\le n+d-1$. Then, $\depth M\ge d-1\ge 1$, so $\m$ is not contained in any associated prime of $M$. By \cite[Corollary 8.5.9]{hs}, choose a superficial element  $x\in \m$ for $\m$ with respect to $N$ and which does not belong to $\bigcup_{\p\in  \Ass(R)\cup\Ass(M)\cup \Min(\sigma^R(N))\smallsetminus \{\m \}}\p$, so that $x$ is $R$ and $M$-regular, so also $N$-regular, and $N/xN$ is Ulrich over $R/xR$ (\cite[Proposition 11.1.9]{hs}).
By Lemma \ref{fs}, $N/xN$ is faithful. 
Applying $\Hom_R(M,-)$ to the exact sequence $0\to N \xrightarrow{x} N \to N/xN\to 0$ gives us $\Ext^i_R(M,N/xN)=0$ for $n\le i\le n+d-2$.
Since $x$ is $M$-regular, so  \cite[Lemma 2, page 140]{mat} gives $\Ext^i_{R/xR}(M/xM,N/xN)=0$ for $n\le i\le n+d-2=n+(d-1)-1$. Now since $\depth M/xM=\depth M-1\ge d-2=\dim R/xR -1$, we get by induction hypothesis that $\pd_{R/xR} M/xM<\infty$, hence $\pd_R M<\infty$.   

(3) We are done by part (2) after noticing that $d-\depth M-1\ge 0$, and $\depth \syz^{d-\depth M-1}_R M\ge d-1$ and $\Ext^i_R(\syz^{d-\depth M-1}_R M, N)=0$ for all $1\le n+1\le i\le n+d$. 

(4) Suppose that $\Ext^i_R(N,M)=0$ for $n\le i\le n+d-1$.
Take a short exact sequence $0 \to Y \to L \to M \to 0$ such that $L$ is maximal Cohen--Macaulay and $\id Y<\infty$ by using Theorem \ref{mcmapp}.
Then it follows that $\Ext^i_R(N,L)=0$ for $n\le i\le n+d-1$ (Remark \ref{mcmapp1}).
Since both $N$ and $L$ are maximal Cohen--Macaulay, it then provides vanishings of $\Ext^i_R(L^\dag,N^\dag)\cong \Ext^i_R(N,L)$ for $n\le i\le n+d-1$.
Remembering $L^{\dag}$ is maximal Cohen--Macaulay and $N^\dag$ is Ulrich and faithful, we now obtain by (2) that $L^\dag$ is free.
It shows that $\id L<\infty$ and hence $\id M<\infty$.
\end{proof}     

\begin{rem}
By \cite[Lemma 2.2]{yo} and Theorem \ref{faithulhigh}(1) it follows that if $N$ is Ulrich and faithful and $M$ is such that there exists an integer $n\ge 1$ such that $\Tor^R_i(M,N)=0$ for all $n\le i\le n+d-1$, then $\Tor^R_i(M,N)=0$ for all $i\ge 1$. Thus, any faithful Ulrich module is $d$-Tor-rigid test module.    
\end{rem}

When the Ulrich module is not faithful but the module we are trying to test finite projective, or injective dimension is locally free at minimal primes, we can still get results similar to Theorem \ref{faithulhigh}. For this, we first record a general preliminary Lemma.  

\begin{lem}\label{minfree} Let $R$ be a local ring of positive depth and dimension $\ge 2$. Let $M$ be an $R$-module locally free on $\Min(R)$. Let $x\in \m$ be an $R$-regular element such that $x\notin \p$ for all $\p\in \Min(\NF_R(M))\smallsetminus\{\m\}$. Then, $M/xM$ is locally free on $\Min(R/xR)$. Moreover, if locally $M$ has the same rank at each minimal prime of $R$, then $M/xM$ has same rank at each minimal prime of $R/xR$.  
\end{lem}  

\begin{proof} We show $\NF_R(M)\cap \Min(R/xR)=\emptyset$. 
Indeed, if not, choose $\p \in \NF(M)\cap \Min(R/xR)$. Then $\Ht \p=1$ as $x$ is $R$-regular. In particular, $\p\neq \m$.
Then by our choice of $x$, $\p\notin \Min(\NF_R(M))$. Since $\NF_R(M)$ is a closed subset and $\p \in \NF_R(M)\smallsetminus \Min(\NF_R(M))$ so choose $\q\in \Min(\NF_R(M))$ contained in $\p$, hence $\q\subsetneq \p$.
But $\Ht \q=0$ as $\Ht \p=1$; this contradicts that $M$ is locally free on $\Min(R)$.
Thus $\NF_R(M)\cap \Min(R/xR)$ is an empty set, hence $M_\p$ is $R_\p$-free for each $\p \in \Min(R/xR)$.
Then we see that $(M/xM)_\p$ is $(R/xR)_\p$-free for each $\p \in \Min(R/xR)$. 
Now if we moreover had assumed that there exists $r$ such that $\rank_{R_\q} M_{\q}=r$ for each $\q\in \Min(R)$, then in the proof above, since we had seen $M_\p$ is $R_\p$-free for each $\p\in \Min(R/xR)$, so choose $\q\in \Min(R)$ such that $\q\subseteq \p$, then $\rank_{R_\p} M_{\p}=\rank_{R_\q} M_{\q}=r$, hence $\rank_{(R/xR)_{\p}}(M/xM)_{\p}=r$.
\end{proof}  

Now using this we can prove the following variant of Theorem \ref{faithulhigh}. 

\begin{thm}\label{ulloc}  Let $R$ be a local Cohen--Macaulay ring of dimension $d\ge 1$. Let $M,N$ be $R$-modules such that $N$ is Ulrich, $\supp(N)=\spec(R)$, and $M$ is locally free on $\Min(R)$. Then the following hold: 
\begin{enumerate}[\rm(1)]
\item 
If there exists an integer $n\ge 1$ such that $\Tor^R_i(M,N)=0$ for all $n\le i\le n+d-1$, then $\pd_R M<\infty$.

\item Assume $\depth M\ge d-1$. If there exists an integer $n\ge 1$ such that $\Ext^i_R(M,N)=0$ for $n\le i\le n+d-1$, then $\pd_R M<\infty$.

\item Assume $\depth M\le d-1$. If there exists an integer $n\ge 0$ such that $\Ext^i_R(M,N)=0$ for $n+d-\depth M \le i\le n+2d-\depth M-1$, then $\pd_R M<\infty$.

\item Assume $\widehat{R}$ is generically Gorenstein. If there exists an integer $n\ge 1$ such that $\Ext^i_R(N,M)=0$ for $n\le i\le n+d-1$, then $\id_R M<\infty$.

\end{enumerate}
\end{thm}

\begin{proof} Note that $\Min(R)=\Ass(R)$ as $R$ is Cohen--Macaulay. Also, $N$ has full support, so $\dim N=d$, so $N$ is maximal  Cohen--Macaulay.
For part (1) and (2), we may pass to a faithfully flat extension of $R$ to ensure that $R$ has infinite residue field.

(1) The proof is similar to that of Theorem \ref{faithulhigh} (1), where the $d=1$ case follows by Lemma \ref{l52}(2), and for the $d\ge 2$ case, we now remember that $M':=\syz_R M$ is locally free on $\Min(R)$, and we choose a superficial element $x\in \m$ for $\m$ with respect to $N$ and which does not belong to $\bigcup_{\p\in\Ass(R)\cup\Min(\NF(\syz_R M))\smallsetminus \{\m \}}\p$ and hence $M'/xM'$ is locally free on $\Min(R/xR)$ by Lemma \ref{minfree}.
Combining with $\supp_{R/xR}(N/xN)=\spec(R/xR)$, we can apply induction.   

(2)   The proof is similar to that of Theorem \ref{faithulhigh} (2), where the $d=1$ case follows by Lemma \ref{l52}(2), and for the $d\ge 2$ case, choose a superficial element  $x\in \m$ for $\m$ with respect to $N$ and which does not belong to $\bigcup_{\p\in\Ass(R)\cup\Ass(M)\cup\Min(\NF(M))\smallsetminus \{\m \}}\p$ and hence $M/xM$ is locally free on $\Min(R/xR)$ by Lemma \ref{minfree}. Combining with $\supp_{R/xR}(N/xN)=\spec(R/xR)$, we can apply induction.   

(3) We are done by part (2) after noticing that $d-\depth M-1\ge 0$, $\syz^{d-\depth M-1}_R M$ is locally free on $\Min(R)$ and $\depth \syz^{d-\depth M-1}_R M\ge d-1$ and $\Ext^i_R(\syz^{d-\depth M-1}_R M, N)=0$ for all $1\le n+1\le i\le n+d$. 

(4) We may pass to completion $\widehat{R}$ of $R$ and assume that $R$ has canonical module and is generically Gorenstein.
Suppose that $\Ext^i_R(N,M)=0$ for $n\le i\le n+d-1$.
Take a short exact sequence $0 \to Y \to L \to M \to 0$ such that $L$ is maximal Cohen--Macaulay and $\id Y<\infty$ by using Theorem \ref{mcmapp}. Then the module $Y$ is locally free on $\Min(R)$ since $R$ is generically Gorenstein, and  $M$ is locally free on $\Min(R)$ by hypothesis. Hence, $L$ is locally free on $\Min(R)$, so $L^{\dagger}$ is locally free on $\Min(R)$. 
Then it follows that $\Ext^i_R(N,L)=0$ for $n\le i\le n+d-1$ (Remark \ref{mcmapp1}).
Since both $N$ and $L$ are maximal Cohen--Macaulay, it then provides vanishings of $\Ext^i_R(L^\dag,N^\dag)\cong \Ext^i_R(N,L)$ for $n\le i\le n+d-1$.
Remembering $L^{\dag}$ is maximal Cohen--Macaulay and $N^\dag$ is Ulrich and $\supp(N^{\dagger})=\supp(N)=\spec(R)$, we now obtain by (2) that $L^\dag$ is free.
It shows that $\id L<\infty$ and hence $\id M<\infty$.  
\end{proof}

For the proof of the next lemma,  see \cite[3.6]{kt}. 

\begin{lem} \label{l53}
Let $R$ be a Cohen--Macaulay local ring and $M$ be a nonfree maximal Cohen--Macaulay $R$-module.
Assume $R$ has minimal multiplicity.
Then $\syz M$ is Ulrich.
\end{lem}  

Now for (deformations of) Cohen--Macaulay local rings of minimal multiplicity, we deduce the following consequences. We note that our vanishing interval always has length one less than the results of \cite{gp}.  

\begin{thm}\label{newmin}
Let $R$ be a Cohen--Macaulay local ring of dimension $d\ge 1$ and $M,N$ be $R$-modules.  
Assume
\begin{enumerate}[ \rm(1)]
\item $R$ has minimal multiplicity,
\item $M$, or $N$  has constant rank, and
\item there exists an integer $t\ge d+2-\max\{\depth M,\depth N\}$ such that $\Tor_i^R(M,N)=0$ for $t\le i\le t+d-1$. 
\end{enumerate}
Then either $\pd M<\infty$ or $\pd_R N<\infty$.  
\end{thm}

\begin{proof} We may assume $\max\{\depth M,\depth N\}=\depth M$ and $\pd M=\infty$ and show $\pd_R N<\infty$.   
Then $\pd M=\infty$ yields that $M':=\syz_R^{d+1-\depth M}M$ is (nonzero) Ulrich.

Case I: $M$ has constant rank: In this case, $M'$ has positive constant rank and hence faithful by Lemma \ref{l27}.  
Then we get $\Tor^R_j(M',N)=0$ for $t-d-1+\depth M\le j\le t-2+\depth M$.
Note that $t-d-1+\depth M\ge 1$, and hence it follows by Theorem \ref{faithulhigh}(1) that $\pd_R N<\infty$. 

Case II: $N$ has constant rank: We proceed  by induction on $d$. For $d=1$, we have $\Tor_j^R(M',N)=0$ for $1\le t-d-1+\depth_R M\le j\le t-2+\depth_R M$. Since $d+1-\depth_R M>0$, so $M'$ is a torsionless non-zero Ulrich module, hence $\grade M'=0$ and $M'\cong \m M'$  (by Lemma \ref{l51}). Thus the assumption that $N$ has constant rank now implies $\pd N<\infty$ by Proposition \ref{t}(6). Now let $d\ge 2$. Choose an $R$-superficial element $x\in \m$ which does not belong to $\bigcup_{\p\in\Ass(R)\cup\Min(\NF(\syz_R N))\smallsetminus \{\m \}}\p$. Since $\syz_R N$ has constant rank, so $\syz_R N/x\syz_R N$ has constant rank by Lemma \ref{minfree}. Since $R$ has minimal multiplicity and $x$ is $R$-superficial, hence $R/xR$ also has minimal multiplicity. Let $M'':=\syz_R^{d-\depth_R M} M$, which is a non-free maximal Cohen--Macaulay $R$-module.  Since $x$ is $R$-regular, so $M''$-regular. Now $\Tor_j^R(M'',N)=0$ for $2\le t-d-\depth_R M\le j\le t-1+\depth_R M$. In view of the exact sequence $0\to M''\xrightarrow{x} M''\to M''/xM''\to 0$ we get $\Tor^R_j(M''/xM'',N)=0$ for $3\le t-d-\depth_R M+1\le j \le t-1+\depth_R M$, so $\Tor_j^R(M''/xM'',\syz_R N)=0$ for $2\le t-d-\depth_R M\le j \le t-2+\depth_R M$. Since $x$ is $\syz_R N$-regular, so by \cite[Lemma 2, page 140]{mat} we get $\Tor_j^{R/xR}(M''/xM'',\syz_R N/x\syz_RN)=0$ for $2\le t-d-\depth_R M\le j \le t-2+\depth_R M$. Now, $M''/xM''$ is maximal Cohen--Macaulay over $R/xR$, so $\depth_{R/xR}M''/xM''=\dim(R/xR)$. Since  $t-d-\depth_R M$ to $t-2+\depth_R M$ are $d-1=\dim(R/xR)$-many elements and $t-d-\depth_R M\ge 2=\dim(R/xR)+2-\depth_{R/xR}(M''/xM'')$, and $\pd_{R/xR} M''/xM''=\infty$ (since $\pd_R M''=\infty$), so by induction hypothesis, $\pd_{R/xR} \syz_R N/x\syz_R N<\infty$, hence $\pd_R \syz_R N<\infty$, so $\pd_R N<\infty$. 
\end{proof}  

\begin{thm}\label{newminext}  Let $R$ be a Cohen--Macaulay local ring of dimension $d\ge 1$ and $M,N$ be $R$-modules.  
Assume
\begin{enumerate}[ \rm(1)]
\item $R$ has minimal multiplicity,
\item $M$  has constant rank, and
\item there exists an integer $t\ge d+2-\depth M$ such that $\Ext^i_R(M,N)=0$ for $t\le i\le t+d-1$.  
\end{enumerate}
Then either $\pd M<\infty$ or $\id_R N<\infty$. 
\end{thm}

\begin{proof}  We may assume $\pd M=\infty$ and show $\id_R N<\infty$.
Then $\pd M=\infty$ yields that $M':=\syz^{d+1-\depth M}M$ is (nonzero) Ulrich.
Moreover, $M'$ has positive constant rank and hence faithful by Lemma \ref{l27}.
Also we get $\Ext_R^j(M',N)=0$ for $t-d-1+\depth M\le j\le t-2+\depth M$.
Note that $t-d-1+\depth M\ge 1$, and hence it follows by Theorem \ref{faithulhigh}(4) that $\id_R N<\infty$. 
\end{proof}

For deformations of Cohen--Macaulay local rings of minimal multiplicity, one gets the following consequence by systematically using the long exact sequence of Tor as described in \cite[Lemma 2.1]{hw1}.  

\begin{cor}\label{defmin} Let $R$ be a local Cohen--Macaulay ring of dimension $d\ge 1$ and minimal multiplicity.
Let $f_1,\cdots,f_c$ be an $R$-regular sequence in $\m$, where $c\ge 1$.
Set $S=R/(f_1,\cdots,f_c)R$.
Let $M,N$ be $S$-modules such that $M$ is maximal Cohen--Macaulay over $S$ and $\Tor^S_i(M,N)=0$ for $t\le i\le t+d+c-1$, where $t\ge 2$ is some integer.
Then, $\Tor^S_i(M,N)=0$ for all $i\ge 1$, and either $\pd_R M<\infty$, or $\pd_R N<\infty$.  
\end{cor}

\begin{proof} 
We prove by induction on $c$.
First we do the $c=1$ case.
In this case, $S=R/(f_1)$ and $\Tor^S_i(M,N)=0$ for $t\le i\le t+d$, where $t\ge 2$ and $M$ is a maximal Cohen--Macaulay $S$-module.
We note that both $M$ and $N$ have constant rank zero over $R$ (since they are killed by the $R$-regular element $f_1$).
The part of the change of rings long exact sequence of Tor: $\Tor^{R/(f_1)}_{n-1}(M,N)\to \Tor^{R}_{n}(M,N)\to \Tor^{R/(f_1)}_n(M,N)$ (for all $n\ge 1$) yields $\Tor^R_i(M,N)=0$ for $t+1\le i\le t+d$.
Hence $\Tor^R_i(\syz_R M,N)=0$ for $2\le t\le i\le t+d-1$.
Since $\depth_R M=\depth_S M\ge \dim S=\dim R-1$, so $\syz_R M$ is maximal Cohen--Macaulay over $R$.
Moreover, $\syz_R M$ has constant rank over $R$.
Thus by Theorem \ref{newmin} either $\pd_R \syz_R M<\infty$, or $\pd_R N<\infty$, i.e., either $\syz_R M$ is free over $R$ (as it is maximal Cohen--Macaulay over $R$), or $\pd_R N<\infty$.  
And then $\Tor^R_i(\syz_R M,N)=0$ for all $i\ge 1$ (because either $\syz_R M$ is free, or $\pd_R N<\infty$, in which case, we apply \cite[Lemma 2.2]{yo}), i.e., $\Tor^R_i(M,N)=0$ for all $i\ge 2$. Again, the part of the change of rings long exact sequence of Tor: $\Tor^R_{n+1}(M,N)\to \Tor^{R/(f_1)}_{n+1}(M,N)\to \Tor^{R/(f_1)}_{n-1}(M,N)\to \Tor^{R}_{n}(M,N)$ (for all $n\ge 1$) yields $\Tor^{R/(f_1)}_{n+1}(M,N)\cong \Tor^{R/(f_1)}_{n-1}(M,N)$ for all $n\ge 2$.
Since $\Tor^{R/(f_1)}_i(M,N)=0$ for some $(d+c)$-many consecutive values of $i\ge 2$ and since $d+c\ge 2$, so we have vanishing for at least two consecutive values, thus we get $\Tor^{R/(f_1)}_i(M,N)=0$ for all $i\ge 1$. 
 
Now for the inductive step: assume $c\ge 2$,  $\Tor^S_i(M,N)=0$ for $t\le i\le t+d+c-1$, where $t\ge 2$ is some integer, and $M,N$ are $S$-modules with $M$ being maximal Cohen--Macaulay over $S$.
Set $S':=R/(f_1,...,f_{c-1})R$.
Then, $S=S'/(f_c)S'$.
The part of the change of rings long exact sequence of Tor: $\Tor^{S'/(f_c)S'}_{n-1}(M,N)\to \Tor^{S'}_{n}(M,N)\to \Tor^{S'/(f_c)S'}_n(M,N)$ (for all $n\ge 1$) yields $\Tor^{S'}_n(M,N)=0$ for $t+1\le n\le t+d+c-1$.
Hence $\Tor^{S'}_n(\syz_{S'} M,N)=0$ for $t\le n\le t+d+c-1-1=t+d+(c-1)-1$. Now $\depth_{S'} M=\depth_S M\ge \dim S=\dim S'-1$.
Thus $\syz_{S'} M$ is a maximal Cohen--Macaulay $S'$-module.
Since $c-1\ge 1$, thus by induction hypothesis applied on $S'$, we get $\pd_R \syz_{S'} M<\infty$, or $\pd_R N<\infty$, and moreover $\Tor^{S'}_i(\syz_{S'}M,N)=0$ for all $i\ge 1$.
So, $\Tor^{S'}_i(M,N)=0$ for all $i\ge 2$.
Since $\pd_R S'<\infty$, so $\pd_R \syz_{S'} M<\infty$ would imply $\pd_R M<\infty$, this concludes the claim about finiteness of projective dimension over $R$.
Now the vanishing $\Tor^{S'}_i(M,N)$ for all $i\ge 2$ and the part of the change of rings long exact sequence of Tor: $\Tor^{S'}_{n+1}(M,N)\to \Tor^{S'/(f_c)S'}_{n+1}(M,N)\to \Tor^{S'/(f_c)S'}_{n-1}(M,N)\to \Tor^{S'}_{n}(M,N)$ (for all $n\ge 1$) yields $\Tor^{S'/(f_c)S'}_{n+1}(M,N)\cong \Tor^{S'/(f_c)S'}_{n-1}(M,N)$ for all $n\ge 2$.
Since $S=S'/(f_c)S'$ and $\Tor_n^{S}(M,N)=0$ for some $(d+c)$-many consecutive values of $n\ge 2$ and $d+c\ge 2$, so we get $\Tor^{S}_i(M,N)=0$ for all $i\ge 1$.
This finishes the inductive step and the proof. 
\end{proof}

\begin{rem} Note that unlike Theorem \ref{newmin}, there is no assumption of constant rank in the hypothesis of Corollary \ref{defmin}. As one can see from the proof, the fact that $d + c \ge 2$, i.e., $c \ge 1$ is crucial for the
argument.  
\end{rem}   

If one requires $1$ more vanishing, namely $(d+c+1)$-many vanishing, then with a simpler proof (and including $c=0$), one gets a similar result to Corollary \ref{defmin}, see \cite[Theorem 5.1]{dg}. For $c\ge 1$, our number of vanishing is more refined than \cite[Theorem 5.1]{dg}. 

The following is an immediate consequence of Corollary \ref{defmin}.  

\begin{cor} \label{519}
Let $R$ be a local Cohen--Macaulay ring of dimension $d\ge 1$ and minimal multiplicity. Let $c\ge 1$ be an integer. Let $f_1,\cdots,f_c$ be an $R$-regular sequence in $\m$. Set $S=R/(f_1,\cdots,f_c)R$. Let $M,N$ be $S$-modules. Let $i_0:=\dim S-\depth M$. Assume $\Tor^S_i(M,N)=0$ for $t\le i\le t+d+c-1$, where $t\ge i_0+2$ is some integer. Then, $\Tor^S_i(M,N)=0$ for all $i\ge i_0+1$, and either $\pd_R M<\infty$, or $\pd_R N<\infty$.   
\end{cor} 

\begin{proof} We apply Corollary \ref{defmin} on $\syz^{i_0}_SM$ and $N$ to get $\Tor^S_i(\syz^{i_0}_S M, N)=0$ for all $i\ge 1$, and moreover, either $\pd_R \syz^{i_0}_S M<\infty$, or $\pd_R N<\infty$. If $\pd_R \syz^{i_0}_S M<\infty$, then $\pd_R S<\infty$ would also give us $\pd_R M<\infty$.  
\end{proof}

The following is an Ext version of Corollary \ref{defmin}.

\begin{cor}\label{defminext} Let $R$ be a local Cohen--Macaulay ring of dimension $d\ge 1$ and minimal multiplicity.
Let $f_1,\cdots,f_c$ be an $R$-regular sequence in $\m$, where $c\ge 1$.
Set $S=R/(f_1,\cdots,f_c)R$.
Let $M,N$ be $S$-modules such that $M$ is maximal Cohen--Macaulay over $S$ and $\Ext_S^i(M,N)=0$ for $t\le i\le t+d+c-1$, where $t\ge 2$ is some integer.
Then, $\Ext_S^i(M,N)=0$ for all $i\ge 1$, and either $\pd_R M<\infty$, or $\id_R N<\infty$.  
\end{cor} 

\begin{proof} The proof is parallel to the proof of Corollary \ref{defmin} by using long exact sequence of Ext (see for instance \cite[Lemma 2.6]{gp}) instead of long exact sequence of Tor, using Theorem \ref{newminext} in place of Theorem \ref{newmin} and using the fact that $\Ext^i_R(X,Y)=0$ for all $i\ge 1$ when $X$ is Maximal Cohen--Macaulay and $Y$ has finite injective dimension in place of \cite[Lemma 2.2]{yo}.    
\end{proof}

If we consider a Cohen--Macaulay local ring of dimension one, we have another source of rigid test property. To see this, we first put a lemma on the structure of syzygy modules with assuming some technical assumption. 

\begin{lem} \label{l521}
Let $R$ be a Cohen--Macaulay local ring  with a canonical module $\omega$ such that there exists a surjective homomorphism $\m^{\oplus n}\to \m^\dag$ for some $n$. Then $\dim R\le 1$. Let $M$ be a nonfree maximal Cohen--Macaulay $R$-module.
Then there exists an $R$-module $L$ such that $\m L\cong (\syz^1 M)^\dag$, so $(\syz^1M)^\dag$ is isomorphic to a Burch submodule of some $R$-module.
If moreover $\dim R=1$, then $\syz^1 M$ is isomorphic to a Burch submodule of some $R$-module.   
\end{lem}   

\begin{proof} If $\dim R\ge 2$, then as in the proof of \cite[Lemma 3.1.(1)]{k}, we see $\m^\dag\cong \omega$. Since $\omega \ne 0$, so by Remark \ref{r25} and Example \ref{mN}, we would get $\omega$ is isomorphic to a Burch submodule of some $R$-module, so by Lemma \ref{burchsum} and its proof, $k$ is a direct summand of $\omega/a\omega$ for some $R$-regular element $a$, contradicting $\depth_{R/aR}\omega/a\omega=\depth_R\omega-1\ge 1$. Thus $\dim R\le 1$. Now let $M$ be a non-free maximal Cohen--Macaulay $R$-module.  Consider a short exact sequence $0 \to \syz^1 M \xrightarrow[]{\iota} F \to M \to 0$ which comes from a minimal free resolution of $M$. Since $\im(\iota)\subseteq \m F$, we get a short exact sequence $0 \to \syz^1 M \to \m F \to \m M \to 0$. If $\dim R=0$, then $\m M$ is maximal Cohen--Macaulay. If $\dim R=1$, then since $\m M$ is a submodule of $M$, hence $\depth \m M>0$, so $\m M$ is maximal Cohen--Macaulay. So in any case, $\m M$ is maximal Cohen--Macaulay, so taking the canonical duals, we obtain a short exact sequence $0\to (\m M)^\dag \to (\m F)^\dag \xrightarrow[]{p} (\syz^1M)^\dag \to 0$. Note that $(\m F)^\dag$ is isomorphic to a direct sum $(\m^{\oplus s})^{\dag}$ of $s=\rank F$ many copies of $\m$. Thus the map $p$ gives a surjection $\m^{\oplus ns} \to (\syz^1 M)^\dag$. Remark \ref{r25} says that $(\syz^1 M)^\dag\cong \m L$ for some $R$-module $L$. Since $(\syz^1 M)^\dagger\ne 0$ (as $M$ is non-free), so by Example \ref{mN}, $(\syz^1 M)^\dag$ is isomorphic to a Burch submodule of some $R$-module. If $\dim R=1$, then by Corollary \ref{313}(2), $\syz^1 M$ is also isomorphic to a Burch submodule of some $R$-module.     
\end{proof}     

Now we obtain two results on Tor rigid test property under the assumption as in Lemma \ref{l521}. In the first one, we assume 2-consecutive vanishing of Tor.
On the other hand, in latter one, we only assume vanishing of Tor at one index, while we restrict one of the considering modules to have constant rank.  

\begin{thm}\label{trig}  
Let $R$ be a Cohen--Macaulay local ring  with a canonical module such that there exists a surjective homomorphism $\m^{\oplus n}\to \m^\dag$ for some $n$. 
Let $M,N$ be $R$-modules.
Assume one of the following holds:

\begin{enumerate}[\rm(1)]
    \item  $\dim R=0$ and $\Tor_{t+1}^R(M,N)=0$ for some integer $t\ge 1$.
    
    \item $\dim R=1$ and  $\Tor_t^R(M,N)=\Tor_{t+1}^R(M,N)=0$ for some $t\ge 3-\max\{\depth M,\depth N\}$.
\end{enumerate}

Then either $\pd M\le 1$ or $\pd N\le 1$.
\end{thm}   

\begin{proof} (1): We may assume $\pd M=\infty$ and aim to show $\pd N \le 1$. By Matlis duality, we have $\Tor_t^R(N,\syz^1 M)^{\dagger} \cong \Ext^t_R(N,(\syz^1 M)^{\dagger})$. As $\Tor^R_t(N,\syz^1 M)\cong \Tor^R_{t+1}(N,M)=0$, so $\Ext^t_R(N,(\syz^1 M)^{\dagger})=0$. By Lemma \ref{l521} we have $(\syz^1 M)^{\dagger}\cong \m L$ for some $R$-module $L$. As $\dim R=0$, and $\syz^1 M\ne 0$, so $\depth (\m L)=0$. Now $\Ext^t_R(N,(\syz^1 M)^{\dagger})=0$ implies $\pd N\le 1$ by Proposition \ref{tt}(3).

(2): Set $s:=\max\{\depth M,\depth N\}$.
We may assume both $\pd M=\infty$ and $s=\depth M$, and aim to show $\pd N\le 1$.
Then $\syz^{1-s} M$ is nonfree and maximal Cohen--Macaulay.
By Lemma \ref{l521}, $\syz^1(\syz^{1-s} M)(=\syz^{2-s} M)$ is a Burch submodule of some $R$-module.
By our assumption, since $t-2+s\ge 1$, so we have vanishings of both $\Tor_{t-2+s}(\syz^{2-s} M, N)$ and $\Tor_{t-1+s}(\syz^{2-s} M, N)$.
Applying Proposition \ref{finpd3} (1), we obtain that $\pd N<\infty$, and then $\pd N\le 1$ by Auslander--Buchsbaum formula.  
\end{proof}

\begin{rem} Let $R$ be Artinian and as in Theorem \ref{trig}. Taking $M$ to be nonfree and maximal Cohen--Macaulay in Theorem \ref{trig}(1), we see that $\Tor^R_{t}(\syz M,N)=0$ for some $t\ge 1$ implies $N$ is free. Thus $\syz M$ is $1$-Tor rigid test. Similarly, when $\dim R=1$, $\syz M$ is $2$-Tor rigid test.   

\end{rem}     

\begin{thm} \label{t55}
Let $R$ be a Cohen--Macaulay local ring of dimension one with a canonical module, and $M,N$ be $R$-modules.
Assume
\begin{enumerate}[\rm(1)]
\item there exists a surjective homomorphism $\m^{\oplus n}\to \m^\dag$ for some $n$,
\item $M$ or $N$ has constant rank, and
\item there exists an integer $t\ge 3-\max\{\depth M,\depth N\}$ such that $\Tor_t^R(M,N)=0$.
\end{enumerate}
Then either $\pd M\le 1$ or $\pd N\le 1$. 
\end{thm}    

\begin{proof}
By symmetry on $M$ and $N$, we may assume $\max\{\depth M,\depth N\}=\depth N$ and call this quantity $s$. 
First we note that $\syz^{1-s} N$ is maximal Cohen--Macaulay and so is $\syz^{2-s}N=\syz(\syz^{1-s} N)$.
In particular, by Lemma \ref{l521}, $(\syz^{2-s} N)^\dag \cong \m L$ for some $R$-module $L$.
Since $t-2+s\ge 1$, so the assumption (3) is equivalent to $\Tor_{t-2+s}(M,\syz^{2-s} N)=0$.
The assumption (2) yields that $\Tor_q(M,\syz^{2-s} N)$ has finite length for all $q>0$.
Thus along with the proof of \cite[3.1]{cd}, one has isomorphisms $\Ext^{1+t-2+s}_R(M,(\syz^{2-s} N)^\dag)\cong \Ext^1_R(\Tor^R_{t-2+s}(M,\syz^{2-s} N),\omega)=0$.
Thus we have a vanishing of $\Ext^{t-1+s}_R(M,\m L)$.
Now we may also assume $\pd N=\infty$, and show $\pd M<\infty$.
We divide the argument in two cases.

Case I: We assume that $N$ has constant rank. Since $\syz^{2-s} N\ne 0$ (as $\pd N=\infty$), so this implies that $\syz^{2-s}N$ has positive constant rank, hence is faithful by  Lemma \ref{l27}. By Lemma \ref{fs} (1), the faithfulness of $\syz^{2-s} N$ implies that of $(\syz^{2-s} N)^\dag(\cong \m L)$ which further implies that of $L$.
Thanks to Proposition \ref{tt} (4), we achieve the inequality $\pd M<\infty$. By the Auslander--Buchsbaum formula, we conclude that $\pd M\le 1$.    

Case II: We assume that $M$ has constant rank.
As $\dim R=1$, it also means that $M$ is locally free on $\spec (R) \smallsetminus\{\m\}$.
Thus by applying Proposition \ref{tt} (7) and remarking that $(\syz^{2-s} N)^\dagger\cong \m L\not=0$, we conclude that $\pd M<\infty$.
The assertion $\pd M \le 1$ is now follows by the Auslander--Buchsbaum formula.    
\end{proof}   

\begin{rem} Taking $M$ to be nonfree and maximal Cohen--Macaulay in Theorem \ref{t55}, we see that $\Tor^R_{t-1}(\syz M,N)=0$ for some $t\ge 2$ implies $\pd N\le 1$, hence $\Tor^R_i(M,N)=0$ for all $i\ge 1$ by \cite[Lemma 2.2]{yo}. Thus $\syz M$ is $1$-Tor rigid test provided that $M$ has constant rank. 
\end{rem}  

We also have the following results.
They can be said Ext versions of Theorem \ref{t55} and Theorem \ref{trig} in some sense.

\begin{thm} \label{t525}
Let $R$ be a Cohen--Macaulay local ring  with a canonical module such that there exists a surjective homomorphism $\m^{\oplus n}\to \m^\dag$ for some $n$.
Let $M,N$ be $R$-modules.
Assume one of the following holds:

\begin{enumerate}[\rm(1)] 

\item $\dim R=0$ and $\Ext^{t+1}_R(M,N)=0$ for some integer $t\ge 1$. 

    \item  $\dim R=1$ and   $\Ext^t_R(M,N)=\Ext^{t+1}_R(M,N)=0$ for some $t\ge 3-\depth M$.
    
\end{enumerate}

Then either $\pd M\le 1$ or $\id N\le 1$.
\end{thm}

\begin{proof} We may assume $\pd M=\infty$, $N\ne 0$ and aim to show $\id N\le 1$.  

 (1): In this case, every module is maximal Cohen--Macaulay. We have isomorphisms of modules $0=\Ext^{t+1}_R(M,N)\cong \Ext^t_R(\syz^1 M,N)\cong \Ext^t_R(N^{\dagger},(\syz^1 M)^{\dagger})$. As $M$ is non-free, so $(\syz^1 M)^{\dagger} \cong \m L$ for some $R$-module $L$ by Lemma \ref{l521}. As $\dim R=0$, and $\syz^1 M\ne 0$ so $\depth (\m L)=0$. Now $N^{\dagger}$ is free by Proposition \ref{tt}(3).  Hence $N$ is a direct sum of copies of $\omega$, so it has finite injective dimension.    

(2): Set $s=\depth M$. Note that $\syz^{1-s} M$ is nonfree and maximal Cohen--Macaulay.
By Lemma \ref{l521}, $\syz^1(\syz^{1-s} M)(=\syz^{2-s} M)$ is a Burch submodule of some $R$-module.
By our assumption, we have vanishings of both $\Ext^{t-2+s}_R(\syz^{2-s} M, N)$ and $\Ext^{t-1+s}_R(\syz^{2-s} M, N)$.
Applying Proposition \ref{321} (remark $t-1+s \ge 1\ge \depth N$), we obtain that $\id N<\infty$, and then $\id N\le 1$ by Bass's formula.
\end{proof}

\begin{rem} Theorem \ref{t525} shows that local Cohen--Macaulay rings  with a canonical module such that there exists a surjective homomorphism $\m^{\oplus n}\to \m^\dag$ for some $n$, satisfy the uniform Auslander's condition (see \cite[pp. 24]{uac}), hence the Auslander-Reiten conjecture holds for such rings, see \cite[Theorem A]{uac}.      
\end{rem}  

\begin{thm} \label{522}
Let $R$ be a Cohen--Macaulay local ring of dimension one with a canonical module, and $M,N$ be $R$-modules.
Assume
\begin{enumerate}[ \rm(1)]
\item there exists a surjective homomorphism $\m^{\oplus n}\to \m^\dag$ for some $n$,
\item either $M$ has constant rank or $R$ is generically Gorenstein and $N$ has constant rank, and
\item there exists an integer $t\ge 3-\depth M$ such that $\Ext^t_R(M,N)=0$.    
\end{enumerate}
Then either $\pd M\le 1$ or $\id N\le 1$. 
\end{thm} 

\begin{proof}
Set $s:=\depth M$.
We may assume $\pd M=\infty$ (hence $\syz^s M\ne 0$) and $N\not=0$, and aim to show $\id N<\infty$.
Note that $\syz^{1-s} M$ is nonfree and maximal Cohen--Macaulay, and $\syz^{2-s} M=\syz^1(\syz^{1-s} M)$.
By Lemma \ref{l521}, $(\syz^{2-s} M)^\dag\cong \m L$ for some $R$-module $L$. By using Theorem \ref{mcmapp}, take a short exact sequence $0 \to Y \to X \to N \to 0$ such that $X$ is maximal Cohen--Macaulay and $\id Y<\infty$.
To see $\id N<\infty$, it is enough to check the inequality $\id X<\infty$.
In view of Remark \ref{mcmapp1} and \ref{214}, we get isomorphisms  
\begin{align*}
\Ext^t_R(M,N)\cong \Ext^{t-2+s}_R(\syz^{2-s}M,N)\cong \Ext^{t-2+s}_R(\syz^{2-s} M,X) &\cong \Ext^{t-2+s}_R(X^\dag,(\syz^{2-s} M)^\dag)\\
&\cong \Ext^{t-2+s}_R(X^\dag,\m L).
\end{align*}
If $M$ has constant rank, then $\syz^{2-s} M$ has positive constant rank, so is faithful by Lemma \ref{l27}. Thus so is $\m L(\cong (\syz^{2-s} M)^\dag)$ by Lemma \ref{fs} (1).
In this case, we may apply Proposition \ref{tt} (4) to deduce $\pd X^\dag<\infty$.
We next suppose that $N$ has constant rank and $R$ is generically Gorenstein.
Then $Y$ also has constant rank, and so does $X$ by the exact sequence $0 \to Y \to X \to N \to 0$.
We again use the assumption that $R$ is generically Gorenstein to see that $X^\dag$ has constant rank.
As $R$ is of dimension one, it yields that $X^\dag$ is locally free on the punctured spectrum $\spec(R)\smallsetminus\{\m\}$.
Thanks to Proposition \ref{tt} (7), we get $\pd X^\dag<\infty$ (We may see that $\m L\not=0$ since $\syz^{2-s} M$ is nonzero).
In any cases, we reach the inequality $\pd X^\dag<\infty$, i.e., $X^\dagger$ is free, which shows $\id X<\infty$.   
\end{proof}

\begin{rem} \label{r56}
Let $R$ be a Cohen--Macaulay local ring of dimension at most one with a canonical module. There exists a surjective homomorphism $\m^{\oplus n}\to \m^\dag$ for some $n$ in the case where either (a) $R$ has minimal multiplicity, or (b)  $\m\cong \m^\dag$.
Indeed, it is clear in the case of (b).
In the case of (a), if $\dim R=1$, then $\m^\dag$ is Ulrich (see \cite[2.5]{bhu}) so that $\m^\dag\cong \m\m^\dag$, hence there exists a surjection $\m^{\oplus n}\to \m^\dag$; and if $\dim R=0$, then $\m^2=0$ implies $\m\cong k^{\oplus \mu(\m)}$, so $\m^\dag\cong (k^\dag)^{\mu(\m)}\cong  k^{\oplus \mu(\m)}\cong \m$, so we are back to case (b). In particular, Theorem \ref{t55} gives another proof of Theorem \ref{newmin} in $1$-dimensional case. It is worth noting that Artinian hypersurfaces satisfies the condition (b). Indeed, for an Artinian hypersurface ring, we have $\m=xR$ for some $x\in \m$. Hence $\m\cong R/\ann_R(x)$, and consequently using that $R$ is Gorenstein, we get $\m^{\dagger}\cong \m^*\cong \ann_R(\ann_R(x))=xR=\m$. We also remark that $1$-dimensional rings with the isomorphism in (b) are studied in \cite{k}. Indeed, if $(R,\m)$ be a Cohen--Macaulay local ring of dimension one with a canonical module and there exists a Gorenstein local subring $(S,\n)$ of $R$ such that $R\cong \End_S(\n)$ as $S$-algebras, then $\m\cong \m^\dagger$ due to \cite[1.4]{k}, so $R$ is generically Gorenstein and the hypothesis (1) of Theorem \ref{t55} and \ref{522} is satisfied. 

It can also be seen that if $R$ is $1$-dimensional, and has decomposable maximal ideal $\m=I_1\oplus I_2$ and $R/I_1, R/I_2$ satisfies the hypothesis of Lemma \ref{l521} (both of these rings are $1$-dimensional Cohen--Macaulay, see for instance the discussion preceding \cite[Lemma 6.14]{burch}), then so does $R$. Indeed, first note that $\m/I_1\cong I_2, \m/I_2\cong I_1$, and $\omega_{R/I_i}\cong \Hom_R(R/I_i, \omega_R)$ (see \cite[Theorem 3.3.7(b)]{bh1}). Thus, by tensor-hom adjunction, we get $\Hom_R(\m/I_i,\omega_R)\cong \Hom_{R/I_i}\left(\m/I_i,\Hom_{R}(R/I_i, \omega_{R})\right)\cong \Hom_{R/I_i}(\m/I_i, \omega_{R/I_i})$, for $i=1,2$. Hence, if $(\m/I_i)^{\oplus n}$ surjects onto $\Hom_{R/I_i}(\m/I_i, \omega_{R/I_i})$, then $\m^{\oplus n} \cong \left(\oplus_{i=1}^2\dfrac \m{I_i}\right)^{\oplus n}$ surjects onto $\oplus_{i=1}^2 \Hom_{R/I_i}(\m/I_i, \omega_{R/I_i})\cong \oplus_{i=1}^2 \Hom_R(\m/I_i, \omega_R) \cong \Hom_R(\m, \omega_R)$.     
\end{rem}

\section{The existence of a surjection from sum of $\m$ to a dual of $\m$}

We further investigate the situation that there exists a surjection $\m^{\oplus n} \to \m^\dag$ for some $n$.
We note the following simple and easy lemma.

\begin{lem}\label{6.1}
Let $R$ be a Cohen--Macaulay local ring of dimension one with a canonical module $\omega$.
Then the following conditions are equivalent.
\begin{enumerate}[ \rm(1)]
\item There exists a surjection $\m^{\oplus n} \to \m^\dag$ for some $n$.
\item $\tau_{\m}(\m^\dag)=\m^\dag$.
\item There exists a surjection $\widehat{\m}^{\oplus n} \to \widehat{\m}^\dag$ for some $n$, where $(\widehat{R},\widehat{\m})$ is the $\m$-adic completion of $R$. 
\end{enumerate}
\end{lem}

\begin{proof}
The equivalence of (1) and (2) is \cite[Lemma 5.1]{kt}. Since the trace module $\tau_{\m}(\m^\dag)$ commutes with the completion, we may see that (2) is equivalent to (3).   
\end{proof}

In view of case (b) of Remark \ref{r56}, it is natural to ask when $\m$ is a homomorphic image of $(\m^{\dagger})^{\oplus n}$ for some $n$.
We first record a lemma that shows under some mild assumptions, such situation occurs exactly when $R$ is nearly Gorenstein. Here we call a Cohen--Macaulay local ring $(R,\m)$ with a canonical module \textit{nearly Gorenstein} if $\tau_\omega(R)$ contains $\m$. Note that the notion of nearly Gorenstein rings are firstly introduced and studied by Herzog, Hibi and Stamate  \cite{hhs}.
Note also that Gorenstein local rings are always nearly Gorenstein.

\begin{lem}\label{near} Let $R$ be a Cohen--Macaulay local ring of dimension at most $1$ with a canonical module. Then, $R$ is nearly Gorenstein if and only if there is a surjection $(\m^{\dagger})^{\oplus n}\to \m$ for some $n$. 
\end{lem} 

\begin{proof} Since $\m,\m^{\dagger}$ are maximal Cohen--Macaulay, so the existence of a surjection as in the hypothesis implies we have an exact sequence $0\to M \to (\m^{\dagger})^{\oplus n}\to \m\to 0$, where $M$ is maximal Cohen--Macaulay. Taking canonical duals, we get an exact sequence $0\to \m^{\dagger} \to \m^{\oplus n}\to M^{\dagger}\to 0$ which implies $R$ is nearly Gorenstein by \cite[Lemma 3.5(4)]{k}.   
Conversely, if we assume $R$ is nearly Gorenstein, then taking canonical duals of the exact sequence of \cite[Lemma 3.5(4)]{k} we get the required surjection.  
\end{proof}

We next attempt to find a relationship between nearly Gorenstein rings and local rings with a surjection $\m^{\oplus n} \to \m^\dag$ for some $n$.
Before stating our result, we need to give a technical lemma.

\begin{lem}\label{episplit} Let $R$ be a henselian local ring. Let $M$ be an indecomposable $R$-module. Then, any surjective $R$-linear map $f: M^{\oplus n} \to M$ is split. 
\end{lem}   

\begin{proof}
Since $M$ is assumed to be indecomposable, it should be nonzero.
We note that the action of the endomorphism ring $A:=\End_R(M)$ of $M$ on $M$ is given by $\phi\cdot x :=\phi(x)$ for all $\phi\in \End_R(M)$ and $x\in M$.
By this action, $M$ can be regard as a left module over $A$.
For each $j=1,\dots, n$, we consider the canonical $j$-th coordinate inclusion map $i_j:M\to M^{\oplus n}$.
For any elements $(x_1,\cdots,x_n)\in M^{\oplus n}$, we have $f(x_1,\cdots,x_n)=\sum_{j=1}^n (f\circ i_j)(x_j)=\sum_{j=1}^n (f\circ i_j)\cdot x_j$.
Now suppose that each $f\circ i_j$ belongs to the Jacobson radical $\mathrm{Jac}(A)$.
Then it yields that $\im(f)\subseteq\mathrm{Jac}(A)\cdot M$, and then surjectivity of $f$ implies $M\subseteq \mathrm{Jac}(A)\cdot M$.
By Nakayama's lemma, we get $M=0$, contradicting the assumption that $M$ is indecomposable.
Thus, there exists $j$ such that $f\circ i_j \notin \mathrm{Jac}(A)$.
Since $R$ is henselian, so Krull–Schmidt holds for $R$-modules, so our assumption that $M$ is indecomposable implies $A$ is a local ring.
In other words, $\mathrm{Jac}(A)$ is the collection of all non-units of $A$.
Therefore $f\circ i_j$ is a unit of $A$, i.e., an automorphism of $M$.
We now achieve an equality $f\circ \left(i_j\circ(f\circ i_j)^{-1}\right)=\text{id}_M$, which shows that $f$ is a split surjection. 
\end{proof}

Now we establish that nearly Gorenstein property along with a surjection $\m^{\oplus n} \to \m^\dag$ for some $n$ forces that $\m$ is self canonical dual under some mild assumptions.

\begin{prop}\label{nind} Let $R$ be a henselian Cohen--Macaulay local ring with a canonical module and indecomposable maximal ideal. Then, the following are equivalent:   

\begin{enumerate}[\rm(1)]
    \item $R$ is nearly Gorenstein and there exists a surjection $\m^{\oplus n} \to \m^\dag$ for some $n$.
    
    \item $\m \cong \m^{\dagger}$.   
\end{enumerate} 
\end{prop}   

\begin{proof} If either (1) or (2) holds, then $\dim R\le 1$ by Lemma \ref{l521}. So, $(2)\implies (1)$ is obvious by Lemma \ref{near}. For $(1) \implies (2)$: By Lemma \ref{near} and our hypothesis, we have two surjections $f:\m^{\oplus n}\to \m^{\dagger}$ and $g:(\m^{\dagger})^{\oplus u}\to \m$ for some integer $n$ and $u$. So we have a surjection $g\circ f^{\oplus u}: \m^{\oplus un}\to \m$. Since $\m$ is indecomposable, so by Lemma \ref{episplit}, the surjection $g\circ f^{\oplus u}$ splits, which mean that there exists a map $h:\m \to \m^{\oplus un}$ such that $g\circ f^{\oplus u}\circ h=\id_{\m}$.
Then $f^{\oplus u}\circ h$ gives a splitting of the surjection $g$, so $\m$ is a direct summand of $(\m^{\dagger})^{\oplus u}$. Since $\m$ is maximal Cohen--Macaulay, $\m^{\dagger}$ is also indecomposable. Thus we must have $\m \cong \m^{\dagger}$.   
\end{proof}

\begin{chunk}\label{colfrac}
Let $R$ be a Cohen--Macaulay local ring.
We denote by $Q(R)$ the total ring of quotients of $R$.
A finitely generated submodule $I$ of $Q(R)$ is called fractional ideal of $R$.
For fractional ideals $I,J$, with $J$ containing some regular element $b$, the colon ideal $I:_{Q(R)}J:=\{a\in Q(R)\mid aJ\subseteq I\}$ is identified with $\Hom_R(J,I)$ via the correspondence $x \mapsto [a\mapsto x\cdot a]$ (and $f\in \Hom_R(J,I)$ goes to $\frac{1}{b}f(b)\in I:_{Q(R)} J$).     

Assume $R$ is one-dimensional.
An ideal $K$ of $R$ is called a \textit{canonical ideal} of $R$ if $K$ is a canonical module of $R$.
Let $K$ be a canonical ideal.
Then $K$ contains a regular element of $R$ and $K:_{Q(R)}K=R$.
Moreover, for any fractional ideal $I$ of $R$ containing a regular element, the equality $K:_{Q(R)}(K:_{Q(R)}I)=I$ holds.
We also note that $R$ has a canonical ideal if and only if $R$ has a canonical module and $Q(R)$ is Gorenstein if and only if $Q(\widehat{R})$ is Gorenstein.
See \cite{hk} for details on canonical ideals.
\end{chunk}

For fractional ideals $I,J$, with $J$ containing a non-zero-divisor $b$, since $f(x)=x\left(\dfrac{1}{b}f(b)\right)$ for every $x\in J$ and $f\in \Hom_R(J,I)$, and since by the discussion of \ref{colfrac}, we have $I:_{Q(R)}J=\left\{\dfrac{1}{b}f(b):f\in \Hom_R(J,I)\right\}$, so the image of the evaluation map $J\otimes_R \Hom_R(J,I)\xrightarrow{x\otimes f\mapsto f(x)} I$ is $J(I:_{Q(R)}J)$. Hence we get the following Lemma   

\begin{lem}\label{6.6}
Let $R$ be a Cohen--Macaulay local ring of dimension one having a canonical ideal $K$.
Then the image of the evaluation map $\m \otimes_R \Hom_R(\m,\m^\dag) \to \m^\dag$ is equal to $\m((K:_{Q(R)}\m):_{Q(R)}\m)$, provided the identification $\m^\dag=K:_{Q(R)}\m$.
\end{lem}

Now we give a characterization of the existence of a surjection $\m^{\oplus n} \to \m^\dag$ over Cohen--Macaulay local rings with a canonical ideal. 

\begin{prop} \label{64}
Let $R$ be a Cohen--Macaulay local ring of dimension one having a canonical ideal $K$.
Then the following conditions are equivalent.
\begin{enumerate}[ \rm(1)]
\item There exists a surjection $\m^{\oplus n} \to \m^\dag$ for some $n$.
\item $\m((K:_{Q(R)}\m):_{Q(R)}\m)=K:_{Q(R)}\m$.
\item $\m^2:_{Q(R)}\m=\m$.
\end{enumerate}
\end{prop}   

\begin{proof}
The equivalence of (1) and (2) follows immediately by Lemma \ref{6.1} and \ref{6.6}.
We now consider the equivalence of (2) and (3).
Applying $K:_{Q(R)}(-)$, we see that the equality $\m((K:_{Q(R)}\m):_{Q(R)}\m)=K:_{Q(R)}\m$ holds if and only if so does $K:_{Q(R)}(\m((K:_{Q(R)}\m):_{Q(R)}\m))=K:_{Q(R)}(K:_{Q(R)}\m)$.
On the other hand, $K:_{Q(R)}(K:_{Q(R)}\m)$ is equal to $\m$, and $K:_{Q(R)}(\m((K:_{Q(R)}\m):_{Q(R)}\m))$ is calculated as follows:
\begin{align*}
K:_{Q(R)}(\m((K:_{Q(R)}\m):_{Q(R)}\m)) &=(K:_{Q(R)}((K:_{Q(R)}\m):_{Q(R)}\m)):_{Q(R)}\m\\
&=(K:_{Q(R)}(K:_{Q(R)}\m^2)):_{Q(R)}\m\\
&=\m^2:_{Q(R)}\m.
\end{align*}
Here the first and second equalities follows by a general fact on colon ideals; for fractional ideals $I,J,L$, one has $(I:_{Q(R)}J):_{Q(R)}L)=I:_{Q(R)}JL$.
Therefore the assertion follows.
\end{proof}

\begin{rem} \label{msq}
Let $R$ be as in the proposition above.
Then $\m:_{Q(R)}\m$ is a module-finite subalgebra of $Q(R)$ which contains both $R$ and $\m^2:_{Q(R)}\m$.
Therefore to see the equality $\m^2:_{Q(R)}\m=\m$, it is enough to check that $f\cdot \m\not\subseteq \m^2$ for any $f\in \m:_{Q(R)}\m\smallsetminus R$.
\end{rem}

\begin{setup}\label{6.9}
Let $k$ be a field.
Let $a_1,a_2,\dots,a_l$ ($l\ge 1$) be positive integers with $\mathrm{gcd}(a_1,\dots,a_l)=1$.
The additive submonoid  
\[
H=\langle a_1,a_2,\dots,a_l\rangle=\left\{\sum_{i=1}^lc_ia_i \mid c_i\in\mathbb{Z}_{\ge 0}\right\}
\]
of $\mathbb{N}$ is called the numerical semigroup generated by $a_i$'s. We always choose the generators $a_1,\dots, a_l$ to be minimal, i.e., $a_i\not\in \langle a_1,\dots,a_{j-1},a_{j+1},\dots a_l\rangle$ for each $j$. We put
\[
R=k[\![t^{a_1},t^{a_2},\dots,t^{a_l}]\!]\subseteq k[\![t]\!],
\]
where $k[\![t]\!]$ be a formal power series ring of one variable over $k$.
Then $R$ is a one-dimensional complete local integral domain.
In particular, $R$ has a canonical ideal.
We denote by $\mathrm{PF}(H)$ the set of pseudo-Frobenius numbers, that is,
\[
\mathrm{PF}(H)=\{a\in\mathbb{N}\smallsetminus H\mid a+a_i\in H \text{ for any }i\}.
\]
See \cite{rg} for details on numerical semigroups.

\end{setup}  

\begin{prop} \label{pfn} Assume the setup of \ref{6.9}. The equality $\m^2:_{Q(R)}\m=\m$ holds if and only if for any $f\in \mathrm{PF}(H)$ there exists indices $i$ and $j$ such that $f+a_i=a_j$.
\end{prop}  

\begin{proof}
Since the $R$-module $\m:_{Q(R)}\m$ is generated by $1$ and monomials $t^a$ for $a\in \mathrm{PF}(H)$, the assertion is a consequence of Remark \ref{msq}.
\end{proof}

\begin{ex} \label{ex}
Let $H=\langle 9,10,61,62 \rangle$ be a numerical semigroup and $R=k[\![H]\!]$ be the associated algebra over a field $k$.
Then one has $\mathrm{PF}(H)=\{51,52,53\}$.
Applying Proposition \ref{pfn}, one can check the equality $\m^2:_{Q(R)}\m=\m$.
Note that $R$ is neither of minimal multiplicity nor nearly Gorenstein.
Thus $\m$ is not self canonical dual by Proposition \ref{nind}. 
\end{ex}

\section{Applications to the conjecture of Huneke and Wiegand}

For an $R$-module $M$, we denote by $\Tr M$ the Auslanader transpose of $M$. Note that $M^*$ is isomorphic to $\syz^2 \Tr M$ up to projective summands.   
We recall the following conjecture on torsion of tensor products made by Huneke and Wiegand \cite{hw1}.

\begin{conj}[Huneke--Wiegand] \label{hw}
If $R$ is a one--dimensional local domain and $M$ be an 
$R$-module such that $M\otimes_R M^*$ is torsion--free, then is $M^*$ free? 
\end{conj}   

Recall that for an $R$-module $M$, the submodule 

$\t(M)=\{x\in M\mid \text{there exists an $R$-regular element $r$ such that } rx=0\}$
of $M$ is called the \textit{torsion part} of $M$.  
Then the factor module $M/\t(M)$ is called the \textit{torsion-free part} of $M$. The following lemma can be found in \cite[2.13]{hw}.
The proof we include below is slightly different from that of \cite[2.13]{hw}.

\begin{lem}\label{100} Let $R$ be a local ring with $\operatorname{depth}(R)\le 1$. If $M$ is a 
torsion-free module such that $M^*$ is free, then $M$ is free.    
\end{lem}

\begin{proof} Enough to show that $M$ is reflexive. Towards this end, to avoid trivialities, let us assume $M\ne 0$. 
Since $M^*$ is free, so it has constant rank, so $M$ has constant rank by \cite[Exercise 1.4.22]{bh1}.
If $M$ has constant rank $0$, then $M\otimes Q(R)=0$, so $M=\t(M)$; but we also assume $M$ is torsion-free, so we would get $M=0$, contradicting what we assume.
Thus $M$ has positive constant rank, and so by \cite[1.4.17]{bh1} $M$ embeds inside a free module, i.e., $M$ is torsion-less.
So, $\Ext_R^1(\Tr M,R)=0$.
Now since $\syz^2 \Tr M\cong M^*$ up to free summands, $\Tr M$ has finite projective dimension over $R$, and so $\pd (\Tr M)=\depth (R)-\depth (\Tr M)\le 1$.
Thus we also get $\Ext_R^2(\Tr M,R)=0$.
Thus $M$ is reflexive.  
\end{proof}

\begin{lem} \label{62}
Let $R$ be a local ring of depth one.
If $M^*$ is free, then the torsion-free part $M/\t(M)$ of $M$ is also free. 
\end{lem}  

\begin{proof}
Since $\depth R>0$, so $\t(M)$ is annihilated by an $R$-regular element $x\in \m$.
Therefore $\t(M)^*=0$ and thus $M^*\cong (M/\t(M))^*$ is free.
Since $M/\t(M)$ is torsion-free, so we are done by Lemma \ref{100}. 
\end{proof}

We achieve the main theorem of this section.
Note that applying (3) of the theorem below to examples like \ref{ex}, we may show that our result certainly provides a new class of rings over which the Question \ref{hw} is affirmative (for known classes we refer \cite{hiw}).  

\begin{thm} \label{74}
Let $R$ be a local ring of depth one.
Let $M$ be an $R$-module such that $\supp(M)=\spec(R)$ and $M$ is locally free on $\Ass(R)$.
Assume moreover that one of the following conditions hold:  
\begin{enumerate} [\rm (1)]
\item there exists an $R$-module $N$ such that $M\cong \m N$.
\item there exists an ideal $J$ of $R$, and an $R$-module $N$ such that $M\subseteq \m J N$, $M:_N J=\m M:_N \m J$, and $N/(M:_N J)$ has finite length.  
\item $R$ is a Cohen--Macaulay local ring with a canonical module $\omega$, $M$ has constant rank and there exists a surjection $\m^{\oplus n} \to \m^\dag$ for some $n$.  
\end{enumerate}  


If $M\otimes_R M^*$ is torsion-free, then $M$ is free, and consequently, in cases $(1)$ and $(2)$, $R$ will be a discrete valuation ring.  
\end{thm}   

\begin{proof}  For both cases (1) and (2), $M$ is a submodule of $N$, hence $\supp(M)\subseteq \supp(N)$, so $\supp(N)=\spec(R)$. For case (2), also notice that $\supp(M)\subseteq \supp(JN)$ implies $\supp(JN)=\spec(R)$. In the rest of the proof, we assume that $M\otimes_R M^*$ is torsion-free and aim to show that $M$ is free.   

Proof of (1) and (2): Since $M^*\cong \syz^2 \Tr M$ up to projective summands, there exist an $R$-module $L$ and a free $R$-module $F$ with a short exact sequence $0\to M^* \to F \to L \to 0$. Since $M$ is locally free on $\Ass(R)$, so $\Tor_1^R(L,M)$ locally vanishes on $\Ass(R)$, hence it is a torsion module. Tensoring $M$ to the sequence, we see that $\Tor_1^R(L,M)$ is a submodule of the torsion-free module $M\otimes_R M^*$, hence $\Tor_1^R(L,M)=0$. Since $M^*$ is locally free on $\Ass(R)$, so $L$ has locally finite projective dimension on $\Ass(R)$, so $L$ is also locally free on $\Ass(R)$ by Auslander-Buchsbaum formula. So now  Proposition \ref{t}(5) or Theorem \ref{t412}(3) implies $\pd L\le 1$. So $\pd M^*\le 1.$ Since $\depth M^*\ge 1=\depth R$, so by Auslander-Buchsbaum formula, $M^*$ is free. Then by Lemma \ref{62}, $M/\t(M)$ is free. Since $M^*\ne 0$ (as $\Ass M^*=\supp(M)\cap \Ass(R)=\Ass(R)\ne \emptyset$) so by \cite[1.1]{hw1}, $M$ is free.  

Proof of (3): Since $M^*\ne 0$, so by \cite[1.1]{hw1}, we may assume that $M$ is torsion-free. Then, since $M$ has constant rank, $M$ is torsionless and hence there exists a short exact sequence $0 \to M \to G \to K \to 0$ of $R$-modules, where $G$ is free, and $K$ has constant rank.
Tensoring $M^*$, it follows that $\Tor_1^R(K,M^*)=0$ as it is both torsion and torsionfree.
It yields that $\Tor_3(K,\Tr M)=0$.
Thus by Theorem \ref{t55}, either $\pd K\le 1$ or $\pd \Tr M\le 1$. As $M$ and $M^*$ is maximal Cohen--Macaulay, it means that $M$ or $M^*$ is free.
In the latter case, it implies that $M$ is free by Lemma \ref{62}.  

Finally, if we are in case (1) or (2) and $M$ is free, then $\Tor^R_i(k,M)=0$ for all $i\ge 1$. As $\depth R=1$, so $\m\notin \Ass(R)$, so $k$ is locally free on $\Ass(R)$. Then, $\pd k \le 1$ by Proposition \ref{t}(5) or Theorem \ref{t412}(3) respectively. As $\depth R=1$, we conclude $R$ is a discrete valuation ring.   
\end{proof}

Theorem \ref{74} have the following two applications.
The former one is already shown in Celikbas, Goto, Takahashi and Taniguchi \cite[Proposition 1.3]{hw};
the latter one is due to Huneke, Iyengar and Wiegand \cite[Corollary 3.5]{hiw}. 

\begin{cor}[Celikbas--Goto--Takahashi--Taniguchi]
Let $R$ be a Cohen--Macaulay local domain of dimension one.
Let $I$ be a nonzero weakly $\m$-full ideal of $R$.
Then $I\otimes_R I^*$ has nonzero torsion if $R$ is not a discrete valuation ring.
\end{cor}

\begin{proof} Taking  $J=N=R$ and $M=I$, we see that the hypothesis of Theorem \ref{74}(2) is satisfied, and moreover $I$ also has positive constant rank. Hence the claim follows.  
\end{proof} 

\begin{cor}[Huneke--Iyengar--Wiegand]
Let $R$ be a Cohen--Macaulay local ring of dimension one and minimal multiplicity.
Let $M$ be an $R$-module of positive constant rank.
Then $M\otimes_R M^*$ has nonzero torsion if $M$ is non-free.  
\end{cor}  

\begin{proof}  This follows from Theorem \ref{74}(3) and Remark \ref{r56}.   
\end{proof}   

In \cite[Question 4.2]{hw}, as a variation of the Huneke-Wiegand conjecture, the authors consider the question whether for local Cohen--Macaulay rings admitting a canonical module, $M\otimes_R M^{\dag}\cong \omega$ implies $M\cong R$ or $M\cong \omega$. In view of the discussion in Remark \ref{r56} about the connection between rings of minimal multiplicity and those satisfying the hypothesis of Lemma \ref{l521}, the following proposition highly generalizes  \cite[Theorem 4.3]{hw}.   

\begin{prop}\label{genhw4.3} Let $R$ be a Cohen--Macaulay local ring with a canonical module $\omega$. Assume that there exists an $R$-regular sequence $x_1,\cdots,x_t$ such that the maximal ideal of $R/(x_1,...,x_t)$ satisfies the hypothesis of Lemma \ref{l521}. Let $M$ be some syzygy (not necessarily in a minimal free resolution) of a maximal Cohen--Macaulay $R$-module. If $N\otimes_R M^\dag \cong \omega^{\oplus n}$ for some integer $n\ge 1$ and some $R$-module $N$, then $M$ and $N^*$ are free. If, moreover, $N$ satisfies Serre's condition $(S_2)$, then $N$ is also free.  
\end{prop}  

\begin{proof} We may assume $R$ is singular. As $n\ge 1$, so $M,N$ are non-zero. Denote $\mathbf x=x_1,\cdots,x_t$ and $\overline{(-)}\coloneqq(-)\otimes_R R/(\mathbf x)$. Let $\m_{\overline R}$ stand for the maximal ideal of $\overline R$. As $R$ is Cohen--Macaulay and singular, so $\overline R$ is singular. As in the proof of \cite[Theorem 4.3.]{hw}, we have $\overline \omega^{\oplus n}\cong \overline N\otimes_{\overline R} \left(\overline M\right)^\dag$. By hypothesis, there is an exact sequence $0\to M \to F \to X \to 0$ for some free $R$-module $F$ and maximal Cohen--Macaulay $R$-module $X$, so $\mathbf x$ is $X$-regular, giving exact sequence $0\to \overline M \to \overline F \to \overline X \to 0$, where $\overline X$ is a maximal Cohen--Macaulay $\overline R$-module.  Hence $\overline M\cong G\oplus \syz_{\overline R}T$, where $T\coloneqq \overline X$ is maximal Cohen--Macaulay over $\overline R$ and $G$ is $\overline R$-free. So, $H:=\overline N\otimes_{\overline R} \left(\syz_{\overline R}T\right)^\dag$ is a direct summand of $\overline \omega^{\oplus n}$, hence $\id_{\overline R} H<\infty$. By Remark \ref{r25} and Lemma \ref{l521}, we have a surjection   $\m_{\overline R}^{\oplus u}\to \left(\syz_{\overline R}T\right)^\dag$ for some integer $u$, which induces a surjection $\overline N\otimes_{\overline R}\m_{\overline R}^{\oplus u}\to \overline N\otimes_{\overline R} \left(\syz_{\overline R}T\right)^\dag=H$. Composing this with the natural surjection  $(\m_{\overline R}^{\oplus u})^{\oplus\mu_{\overline R}(\overline N)}\to \overline N\otimes_{\overline R}\m_{\overline R}^{\oplus u}$, we see that direct sum of some copies of $\m_{\overline R}$ surjects onto $H$, hence either $H=0$, or $H$ is isomorphic to a Burch submodule of some module by Remark \ref{r25} and Example \ref{mN}. Since $\overline R$ is singular and $\id_{\overline R} H<\infty$, so $H$ cannot be isomorphic to a Burch submodule by Corollary \ref{bpd}. Thus $0=H=\overline N\otimes_{\overline R} \left(\syz_{\overline R}T\right)^\dag$, hence $\left(\syz_{\overline R}T\right)^\dag=0$, so $\syz_{\overline R}T=0$, so $\overline M\cong G$ is $\overline R$-free. As $\mathbf x$ is $M$-regular, we get $M$ is $R$-free, say of rank $s$. Thus, $N\otimes_R \omega^{\oplus s}\cong \omega^{\oplus n}$, so $R^{\oplus n}\cong \Hom_R(N\otimes_R \omega^{\oplus s},\omega)\cong \Hom_R\left(N,\End_R(\omega)^{\oplus s}\right)\cong \left(N^*\right)^{\oplus s}$, where the one before the last isomorphism follows from Hom-Tensor adjunction. Thus $N^*$ is free. If, moreover, $N$ satisfy $(S_2)$, then $N$ is free by \cite[Theorem 3.10]{dl}.  
\end{proof}    

\section{Examples}
This section is devoted to collect some examples that complement our results.
We start by showing that faithfulness of $N$ in Theorem \ref{faithulhigh}
is a necessary assumption. 

\begin{ex}
Let $(Q, \mathfrak n)$ be a regular local ring of dimension $2$. Let $f, g \in \mathfrak n\smallsetminus \mathfrak n^2$. 
So, $\text{ord}(fg)=\text{ord}(f)+\text{ord}(g)=2$ (by \cite[Theorem 6.7.8]{hs}). Then, $R=Q/(fg)$ is a local hypersurface ring of dimension $1$ and minimal multiplicity (by \cite[Example 11.2.8]{hs}). Now also assume $f,g$ are relatively prime in $Q$. Then, the image of $f$ and $g$ in $R$, call them $x,y$, are orthogonal exact pair of zero-divisors in $R$ by \cite[Example 2.4]{holm}. So, $R/xR, R/yR$ are totally reflexive modules, hence are $n$-th syzygy modules for any $n$, and moreover they have no free-summand, hence $R/xR$ and $R/yR$ are Ulrich $R$-modules by \cite[3.6]{kt}. Now, $\Tor_1^R(R/xR,R/yR)=0$ by \cite[Lemma 2.2]{holm}. Moreover, since $R\xrightarrow{y}R\xrightarrow{x}R\to R/xR\to 0$ is a presentation of $R/xR$ over $R$, so $\Ext^1_R(R/xR, R/xR)=\ann_{R/xR}(y)$. But $\ann_{R/xR}(y)=0$ since $y$ is regular on $R/xR$ by \cite[Lemma 2.2]{holm}, hence $\Ext^1_R(R/xR, R/xR)=0$. But $R/xR$ is not free, hence has infinite projective dimension by Auslander-Buchsbaum formula.   
\end{ex}

The following example shows that the assumption $N\subseteq \m X$ in Corollary \ref{th411} is indispensable in general.

\begin{ex}
Keeping notations as in the previous example.
We remark that $\m/yR\cong R/yR$ since $R/yR$ is a discrete valuation ring. Hence $\m/yR$ is a weakly $\m$-full submodule of $R/yR$. 
Consider the $R$-module $X=R/yR\oplus R$ and the submodule $N=\m/yR\oplus R$ of $X$. Then one can see that $N$ is a weakly $\m$-full submodule of $X$.
Moreover, $X$ is faithful as it has a nonzero free summand, $X/N$ has finite length and $\Tor_1^R(R/xR,N)=0$. But we have $\pd M=\infty$.  
\end{ex}

It is natural to ask whether a Burch submodule of a faithful module is rigid test or not.
We give a negative answer by posing the following example.
 
\begin{ex}
Let $R=k[\![X,Y,Z,W]\!]/(X^3,X^2Y,XY^2,Y^3,XW)$, a complete local algebra over a field $k$.
Denote by $x,y,z,w$ the images of $X,Y,Z,W$ in $R$.
Consider an $R$-module $M=R/(x^2)$ and a submodule $N=(xy,y^2,z,w)$ of $X:=R$.
Then one has $y\in N:_R \m$ and $xy\not\in \m N$, so $\m(N:_X \m)\not=\m N$.
Thus $N$ is a Burch submodule of $X$.
We see a presentation $R^{\oplus 2}\xrightarrow[]{[x,y]}R \xrightarrow[]{x^2} R \to M \to 0$ of $M$ over $R$.
Tensoring $X/N$, we get a complex $X/N^{\oplus 2}\xrightarrow[]{[x,y]}X/N \xrightarrow[]{x^2} X/N$ of $R$-modules whose homology gives $\Tor_1^R(M,X/N)$.
Then one can compute that it is exact, i.e., $\Tor_1^R(M,X/N)=0$.
On the other hand, $\pd M=\infty$ and thus $\Tor_2(M,X/N)\not=0$ by Proposition \ref{finpd3}. Moreover, one may check that $M=\syz^1(R/wR)$ and $N=\syz^1(X/N)$.
It then implies that $\Tor_1(R/wR,N)=\Tor^R_2(R/wR,X/N)=\Tor_1(M,X/N)=0$. Similarly, $\Tor_2(R/wR,N)=\Tor_2(M,X/N)\not=0$. We conclude that $N$ is not 1-Tor-rigid. 
\end{ex}

\begin{ac}
We thank Professor Olgur Celikbas for helpful comments and for informing us of \cite{clty}. We are grateful to the anonymous referee for several helpful comments to improve the presentation of the paper.   
\end{ac}

\end{document}